\documentclass[12pt,a4paper]{article}
\usepackage{authblk}
\usepackage[french,english]{babel}
\usepackage[utf8]{inputenc}
\usepackage[T1]{fontenc}
\usepackage{amsmath}
\usepackage{amsfonts}
\usepackage{amssymb}
\usepackage{hyperref,enumitem}
\hypersetup{
    colorlinks=true,
    urlcolor=blue,
    linkcolor=blue,
    breaklinks=true
}
\usepackage{url}
\usepackage{xcolor,color}
\colorlet{darkgreen}{green!50!black}
\usepackage{amsbsy}
\usepackage{amsmath}
\usepackage{dsfont}

\usepackage{geometry}
\usepackage{graphicx}

\usepackage[numbers]{natbib}

\newcommand{\hatpi}{\widehat{\boldsymbol{\pi}}}
\newcommand{\hatnu}{{\widehat{\boldsymbol{\nu}}}}
\newcommand{\hatf}{\widehat{\boldsymbol{f}}}
\newcommand{\boldpi}{{\boldsymbol{\pi}}}
\newcommand{\boldnu}{{\boldsymbol{\nu}}}
\newcommand{\D}{{\boldsymbol{\mathrm{D}}}}

\renewcommand{\ge}{\geqslant}
\renewcommand{\geq}{\geqslant}
\renewcommand{\leq}{\leqslant}

\usepackage{amsthm}
\usepackage{oubraces}

\usepackage{bm}

\newtheorem{theorem}{Theorem}

\newtheorem{proposition}{Proposition}
\newtheorem{corollary}[proposition]{Corollary}
\newtheorem{remark}[proposition]{Remark}
\newtheorem{lemma}[proposition]{Lemma}

\newtheorem{statement}[proposition]{Statement}

\author[$1$]{Sandro Franceschi
}
\author[$2$]{Tomoyuki Ichiba}
\author[$3$]{Ioannis Karatzas}
\author[$4$]{Kilian Raschel}

\affil[$1$]{T\'el\'ecom SudParis, Institut Polytechnique de Paris}
\affil[$2$]{Department of Statistics and Applied Probability, University of California Santa Barbara}
\affil[$3$]{Department of Mathematics, Columbia University}
\affil[$4$]{University of Angers, CNRS}

\title{Invariant measure of gaps in degenerate competing three-particle systems}
\date{\today}

\begin{document}

\maketitle
\thispagestyle{empty}

\abstract{We study the gap processes in a degenerate system of three particles interacting through their ranks. We obtain the Laplace transform of the invariant measure of these gaps, and an explicit expression for the corresponding invariant density. To derive these results, we start from the basic adjoint relationship characterizing the invariant measure, and apply a combination of two approaches: first, the invariance methodology of W. Tutte, thanks to which we compute the Laplace transform in closed form; second, a recursive compensation approach which leads to the density of the invariant measure as an infinite convolution of exponential functions. As in the case of  Brownian motion with reflection or killing at the endpoints of an interval, certain Jacobi theta functions play a crucial role in our computations.
}

\section{Introduction and main results}

\subsection{Degenerate competing three-particle systems} \label{sec:1.1}
The paper \cite{ichiba_karatzas_degenerate_22} studies degenerate three-particle systems of Brownian particles, in which local characteristics are assigned by rank. Among these is the system
\begin{equation*}
X_i(\cdot)=x_i + \sum_{k=1}^3 \delta_k \int_0^\cdot \mathds{1}_{X_i(t)=R_k^X (t)} \mathrm{d}t+ \int_0^\cdot \mathds{1}_{X_i(t)=R_2^X (t)} \mathrm{d}B_t, \quad i=1,2,3,
\end{equation*}
with the notation $\max_{1\leqslant i \leqslant 3} X_i(t)=:R_1^X \geqslant R_2^X(t)\geqslant R_3^X(t):= \min_{1\leqslant i \leqslant 3} X_i(t)$
for the ranks (order statistics)\ in descending order; with ``lexicographic'' resolution of ties, i.e., always in favor of the lowest index $i$; with $x_i$ and $\delta_i$ given real numbers; and with $B_1(\cdot),B_2(\cdot),B_3(\cdot)$ independent scalar Brownian motions.

It is shown in \cite{ichiba_karatzas_degenerate_22} that this system admits a pathwise unique, strong solution, which is free of triple collisions as well as ``non-sticky'', in the sense
\begin{equation*}
   \int_0^\infty \mathds{1}_{R_k^X(t)=R_\ell^X(t)}\mathrm{d}t=0\quad
\text{for } k<\ell.
\end{equation*}
It is also shown that the two-dimensional process
\begin{equation}
\label{eq:gapdef}
\bigl(G(\cdot ),H(\cdot )\bigr)\overset{\Delta}{=} \bigl(R_1^X (\cdot )-R_2^X (\cdot ),R_2^X (\cdot )-R_3^X (\cdot )\bigr)
\end{equation}
is a degenerate Brownian motion in the nonnegative orthant $[0,\infty)^2$ with oblique reflection on its boundaries:
\begin{equation}
\label{eq:gapSRBM}
\begin{cases}
G(t)=x_1-x_2+(\delta_1-\delta_2)t-W(t)-\frac{1}{2}L^H(t)+L^G(t),
\\
H(t)=x_2-x_3+(\delta_2-\delta_3)t+W(t)-\frac{1}{2}L^G(t)+L^H(t),
\end{cases}
\end{equation}
for $0\leqslant t < \infty$. We denote here by $W(\cdot)$ a suitable standard, scalar Brownian motion, and by $L^Z(\cdot)=\int_0^\cdot \mathds{1}_{Z(t)=0}\mathrm{d}Z(t)$ the local time at the origin of a semimartingale $Z(\cdot ) \geqslant 0$ with continuous paths.

It is shown in \cite[Thm~2.3]{ichiba_karatzas_degenerate_22} that, under the \citet{hobson_recurrence_1993} conditions
\begin{equation}
2(\delta_3-\delta_2)+(\delta_1-\delta_2)^->0\quad \text{and}\quad
2(\delta_2-\delta_1)+(\delta_2-\delta_3)^->0,
\label{eq:recurrence_condition}
\end{equation}
the process $\bigl(G(\cdot ),H(\cdot )\bigr)$ of \eqref{eq:gapdef}--\eqref{eq:gapSRBM} is positive recurrent and has a unique invariant measure $\boldpi$ with $\boldpi\bigl((0,\infty)^2\bigr)=1$, to which its time-marginal distributions converge, and exponentially fast, as $t\to\infty$.

This invariant probability measure $\boldpi$ satisfies, in fact is characterized by, the so-called ``Basic Adjoint Relationship'' (BAR) of \cite{harrison_brownian_1987,KuSt-01}. This involves also the ``lateral measures''
\begin{equation}
\label{eq:def:nu1}
\boldnu_1 (A) \overset{\Delta}{=}  \mathbb{E}^\boldpi \int_0^2 \mathds{1}_A(H(t)) \mathrm{d} L^G(t)
\quad\text{and}\quad
\boldnu_2 (A) \overset{\Delta}{=} \mathbb{E}^\boldpi \int_0^2 \mathds{1}_A(G(t)) \mathrm{d}L^H(t)
\end{equation}
for $A\in\mathcal{B}\bigl((0,\infty)\bigr)$, and is cast most concisely as
\begin{equation} 
\label{eq:funceq}
\left[(x-y)^2+2(\delta_2-\delta_1)x+2(\delta_3-\delta_2)y \right]\hatpi (x,y)
=\left(x-\frac{y}{2}\right)\hatnu_1(y)+\left(y-\frac{x}{2}\right)\hatnu_2(x)
\end{equation}
for $(x,y)\in [0,\infty)^2$, in terms of the Laplace transforms
\begin{align}
\hatpi (x,y) &\overset{\Delta}{=} \mathbb{E}^\boldpi \left( e^{-xG(t)-yH(t)} \right)=\iint_{(0,\infty)^2}e^{-xg-yh}\boldpi(\mathrm{d}g,\mathrm{d}h),
\label{eq:def:laplpi}\\
\hatnu_1(y)&\overset{\Delta}{=}\int_0^\infty e^{-yu}\boldnu_1(\mathrm{d}u)=\lim_{x\to\infty}x\hatpi(x,y),
\label{eq:def:laplnu1}\\\label{eq:def:laplnu2}
\hatnu_2(x)&\overset{\Delta}{=}\int_0^\infty e^{-xu}\boldnu_2(\mathrm{d}u)=\lim_{y\to\infty}y\hatpi(x,y).
\end{align}
Conversely, a probability measure ${\bm \pi}$ on $\mathcal B( (0, \infty)^2) $ is invariant for the process $(G(\cdot), H(\cdot))$ of gaps if it, together with two finite measures ${\bm \nu}_1, {\bm \nu}_2$ on $\mathcal B ((0, \infty))$, satisfies  the BAR of~\eqref{eq:funceq}.  

In the arXiv version \cite{ichiba_karatzas_degenerate_22-arxiv} of the paper \cite{ichiba_karatzas_degenerate_22}, the following question was raised: \emph{
Can the invariant probability measure $\boldpi$ of the two-dimensional process
$\bigl(G(\cdot), H(\cdot)\bigr)$ of gaps be computed explicitly?} See \cite[Sec.~2.4]{ichiba_karatzas_degenerate_22-arxiv}.
The goal of the present paper is to answer this question. 

\subsection{Main results}


In what follows, we set
\begin{equation}
\label{eq:def_lambda_1_2}
\lambda_1 := 2(\delta_2-\delta_1)
\quad\text{and}\quad
\lambda_2 := 2(\delta_3-\delta_2).
\end{equation} 
We will impose the ergodicity conditions 
introduced in \eqref{eq:recurrence_condition}; in terms of the quantities $\lambda_1$ and $\lambda_2$ defined in \eqref{eq:def_lambda_1_2}, they can be cast as
$\lambda_2>\frac{1}{2} \lambda_1^+$ and 
 $ \lambda_1>\frac{1}{2} \lambda_2^+ $.
Again as in \cite{ichiba_karatzas_degenerate_22}, and in order to restrict the number of cases to handle, we will impose the stronger condition 
\begin{equation*}
\delta_1<\delta_2<\delta_3,
\end{equation*}
which is equivalent to
\begin{equation}
\label{eq:conddelta}
\lambda_1>0
\quad\text{and}\quad
\lambda_2>0.
\end{equation}
The symmetric case refers to the assumption 
\begin{equation}
\label{eq:symcond}
   \lambda_1 = \lambda_2 = 2(\delta_2-\delta_1)=2(\delta_3-\delta_2)=:\lambda > 0. 
\end{equation}

\noindent {\bf A):} The first main result of this paper gives a simple explicit expression for the Laplace transform of the invariant distribution. 
\begin{theorem}[Laplace transform, general case]
\label{thm:explicit_general}
The Laplace transform \eqref{eq:def:laplnu1} of the lateral measure $\boldnu_1$ in \eqref{eq:def:nu1} is given by
\begin{equation}
\label{eq:expression_hatnu_1_nonsym}
   \hatnu_1({y})=\frac{4\pi}{3\lambda_1\lambda_2}\sin\left(\pi\frac{\lambda_1}{\lambda_1+\lambda_2}\right)\frac{y(y+\lambda_2)(y+2\lambda_1+\lambda_2)}{
   \cos \left( \pi \sqrt{\frac{\lambda_1^2}{(\lambda_1+\lambda_2)^2}-\frac{4y}{\lambda_1+\lambda_2}} \right)
   -\cos \left( \pi \frac{\lambda_1}{\lambda_1+\lambda_2}   \right)}.
\end{equation}
\end{theorem}

Exchanging the variables $x \leftrightarrow y$ and the parameters $\lambda_1\leftrightarrow \lambda_2$, we derive a similar expression for the Laplace transform $\hatnu_2({x})$ in \eqref{eq:def:laplnu2}. The bivariate Laplace transform $\hatpi(x,y)$ in \eqref{eq:def:laplpi} is then obtained via the main equation~\eqref{eq:funceq}.

As a Laplace transform, the function $\hatnu_1$ is analytic in the half-plane with negative real part. Since the function $\cos \sqrt{z} = \sum_{n\geq 0} \frac{(-z)^n}{(2n)!}$ is analytic on $\mathbb C$, an immediate consequence of Theorem~\ref{thm:explicit_general} is that $\hatnu_1$ of \eqref{eq:expression_hatnu_1_nonsym} admits a meromorphic continuation to the whole of $\mathbb C$. A globally meromorphic infinite product representation of $\hatnu_1$ will be given in Section~\ref{sec:sum_exp_sum}, see~\eqref{eq:dist_nu_i_inf_conv}.

In the symmetric case, the above result simplifies as follows:
\begin{corollary}[Laplace transform, symmetric case]
\label{thm:explicit_symmetric}
Assuming \eqref{eq:symcond}, 
the Laplace transform \eqref{eq:def:laplnu1} of the lateral measure $\boldnu_1$ in \eqref{eq:def:nu1} is given by
\begin{equation}
\label{eq:expression_hatnu_1_sym}
\hatnu_1({y})=\frac{4\pi}{3\lambda^2}\frac{y(y+\lambda)(y+3\lambda)}{\cos \left(\frac{\pi}{2} \sqrt{1-\frac{8y}{\lambda}}\right)}.
\end{equation}
\end{corollary}


While the above results provide fairly simple expressions for the Laplace transforms $\hatnu_1({y})$, $\hatnu_2({x})$ and $\hatpi(x,y)$ of the marginal and the joint distributions, they do not address the question of finding the associated density functions in closed form. This is the topic of our subsequent results; we shall actually propose two ways to compute these densities. 

\medskip

\noindent {\bf B):} The first method appears as a consequence of Theorem~\ref{thm:explicit_general} and Corollary~\ref{thm:explicit_symmetric}: classical Mittag-Leffler expansions allow us to express the trigonometric Laplace transforms \eqref{eq:expression_hatnu_1_nonsym} and \eqref{eq:expression_hatnu_1_sym} as infinite sums, each term of which may be interpreted as the Laplace transform of an exponential term. 

Introduce the parameters 
\begin{equation}
\label{eq:def_mu}
   \mu_1=\frac{\lambda_1 }{\lambda_1  + \lambda_2 }
   \quad \text{and}\quad
   \mu_2=\frac{\lambda_2 }{\lambda_1  + \lambda_2 } = 1-\mu_1,
\end{equation}
which belong to $(0,1)$ due to our hypothesis \eqref{eq:conddelta}.
\begin{theorem}[Density on the boundary, general case]
\label{cor:boundary_density_non-sym}
For $i\in\{1,2\}$, the density function $\nu_i$ of the measure $\boldnu_i$ in \eqref{eq:def:nu1} is equal to
\begin{multline}
\label{eq:formula_boundary_density_original_question_non_symmetric}
    \nu_i(u) =\\ 
    -\frac{4(\lambda_1+\lambda_2)^4}{3\lambda_1\lambda_2}
    \sum_{n\in\mathbb Z}(n-1)n(n+1)(n - 1 + \mu_i)
    (n + \mu_i)(n + 1 + \mu_i )(n + \mu_i /2)e^{-(n^2+\mu_i n)(\lambda_1+\lambda_2)u}.
\end{multline}
\end{theorem}

See Figure~\ref{fig:graph_1D} for an example of graph of $\nu_i$. 

Theorem~\ref{cor:boundary_density_non-sym} has an interesting reformulation in terms of a certain Jacobi theta-type function, namely
\begin{equation}
\label{eq:def_theta_mu}
   \theta_{\mu_i}(q) := \sum_{n\in\mathbb Z} \Bigl(n+\frac{\mu_i}{2}\Bigr) q^{n(n+\mu_i)}, \quad \vert q\vert <1,
\end{equation}
which is intimately related to our model and has a direct probabilistic interpretation in terms of Brownian motion conditioned to stay in an interval, see \eqref{eq:probab_interpretation_theta_mu} in Appendix~\ref{sec:app_theta_mu}. More precisely, introducing the differential operator 
\begin{equation}
\label{eq:diff_op}
    \D_1[f] = f'''+2(\lambda_1 + \lambda_2) f''+\lambda_2(\lambda_2 + 2\lambda_1) f',
\end{equation}
we shall show that
\begin{equation} \label{eq: nu1(u)}
    \nu_1(u)= \frac{4(\lambda_1+\lambda_2)}{3\lambda_1\lambda_2}  \D_1 \left[ \theta_{\mu_1}\left(e^{-(\lambda_1+\lambda_2) \boldsymbol{\cdot} }\right) \right](u).
\end{equation}
A similar expression holds for $\nu_2$.
Notice that the differential operator \eqref{eq:diff_op} corresponds to the polynomial $y(y+\lambda_2)(y+2\lambda_1+\lambda_2)$ appearing in the formula \eqref{eq:expression_hatnu_1_nonsym} of Theorem~\ref{thm:explicit_general}, in the sense that 
\begin{equation*}
   \D_1 = \frac{\mathrm{d}}{\mathrm{d}u}\left(\frac{\mathrm{d}}{\mathrm{d}u}+\lambda_2\right)\left(\frac{\mathrm{d}}{\mathrm{d}u}+2\lambda_1+\lambda_2\right).
\end{equation*}
We will elaborate on this connection in Section~\ref{sec:boundary_density}.

Theorem~\ref{cor:boundary_density_non-sym} also contains the case of equal parameters $\lambda_1=\lambda_2$, corresponding to $\mu_1=\mu_2=\frac{1}{2}$. However, in this symmetric case, it is natural to reformulate the bi-infinite summation \eqref{eq:formula_boundary_density_original_question_non_symmetric} as a sum over the positive integers, using natural symmetries. More precisely, one has:
\begin{corollary}[Density on the boundary, symmetric case]
\label{cor:boundary_density_sym}
If $\lambda_1=\lambda_2=\lambda$, for $i\in\{1,2 \}$ the density function $\nu_i$ of the measure $\boldnu_i$ in \eqref{eq:def:nu1} is equal
to
\begin{align}
\label{eq:formula_boundary_density_original_question}
   \frac{\nu_i(u)}{\lambda^2} &=\sum_{n\geq 3}\frac{(-1)^{n-1}}{12}(n-2)(n-1)n(n+1)(n+2)(n+3)(2n + 1)\exp\left(-\frac{n(n+1)}{2}\lambda u\right)\\\nonumber
   &=420\, \Bigl( e^{-6\lambda u}-9e^{-10\lambda u}+44e^{-15\lambda u}- 156e^{-21\lambda u}+ 450e^{-28\lambda u} -\ldots \Bigr)\, , \, \, \, i = 1, 2\, . 
\end{align}
\end{corollary}
In the context of a two-queue fluid polling model and the corresponding two-dimensional degenerate Brownian motion reflected normally on the boundary of the nonnegative orthant $[0, \infty)^2$, the recent paper \cite{kapodistria_saxena_boxma_kella_2023} analyses a functional equation quite similar to ours. The main equation (see \cite[Eq.~(14)]{kapodistria_saxena_boxma_kella_2023}) of that paper corresponds to \eqref{eq:funceq}, if the prefactors $(x-\frac{y}{2})$ and $(y-\frac{x}{2})$ are replaced by $x$ and $y$, respectively. The authors obtain various results, such as the Laplace transform of the total workload (their Theorem~2, which is close to our Corollary~\ref{thm:explicit_symmetric}, assuming symmetry of the parameters), and the heavy-traffic stationary workload distribution (their Lemma~4, which resembles our Corollary~\ref{cor:boundary_density_sym}).

\medskip

\begin{figure}[hbtp]
\centering
\vspace{-2mm}
\includegraphics[scale=0.3]{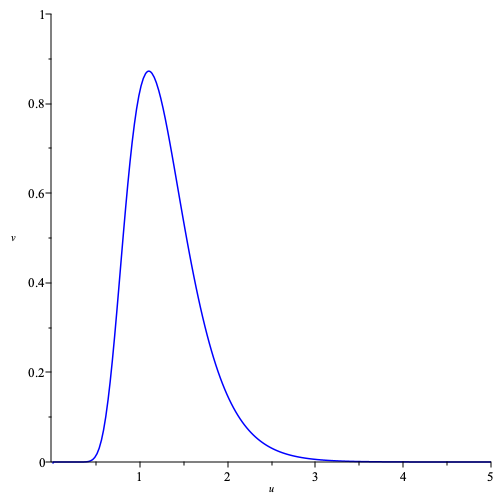}\qquad \qquad 
\includegraphics[scale=0.3]{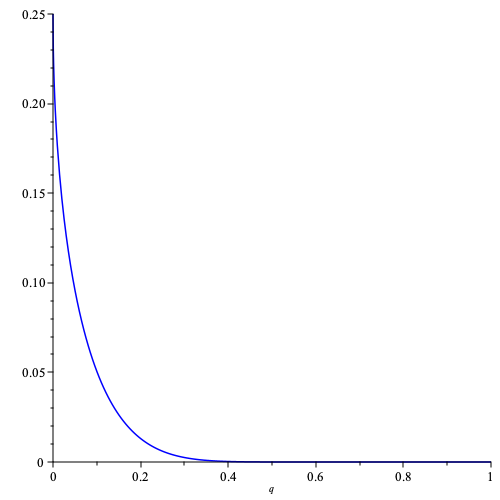}
\vspace{-2mm}
\caption{Left: graph of the function $\nu_2(u)$ for the parameters $(\lambda_1,\lambda_2)=(\frac{1}{6},\frac{5}{6})$. Right: graph of the function $\theta_{\mu}(q)$ in the case $\mu=\frac{1}{2}$.}
\label{fig:graph_1D}
\end{figure}

\noindent {\bf C):} The second approach allowing us to calculate the densities is called the ``compensation approach''; it is very different and brings two advantages. The first advantage is that it does not require a bivariate Laplace inversion. The second advantage is that it works directly for the bivariate density function, without recourse to the univariate boundary density functions. This approach is inspired by the paper \cite{adan_wessels_zijm_compensation_93}, which proposed a compensation methodology for computing the stationary distribution of certain singular random walks in the positive quarter-plane. While this technique has been applied to a variety of contexts in discrete probability, our paper contains its first application to diffusions, to the best of our knowledge. 

\begin{theorem}[Density of the invariant measure, general case]
\label{thm:maindensity}
The density $\pi(u,v)$ of the invariant measure $\boldpi$ satisfies for $(u,v)\in[0,\infty)^2$
\begin{equation} 
\label{eq: thm 3} 
   \pi(u,v)=\sum_{n=0}^\infty \left(Cc_n e^{-a_nu-b_nv}+C'c'_ne^{-a'_nu-b'_nv}\right),
\end{equation}
where
\begin{itemize}
    \item the sequences $(a_n,b_n)_{n\geq0}$ and $(a'_n,b_n')_{n\geq0}$ are given in \eqref{eq:values_(a_n,b_n)} and \eqref{eq:values_(ap_n,bp_n)},
    \item the constants $C$ and $C'$ are computed in \eqref{eq:choice_C_C'},
    \item the $(c_n)_{n\geq0}$ are defined in \eqref{eq:values_c_n_even}--\eqref{eq:values_c_n_odd}, and the $(c'_n)_{n\geq0}$ are obtained from the $(c_n)_{n\geq0}$ after interchanging $\lambda_1$ and $\lambda_2$.
\end{itemize} 
\end{theorem}

\begin{figure}[hbtp]
\centering
\vspace{-6.5mm}
\includegraphics[scale=0.3]{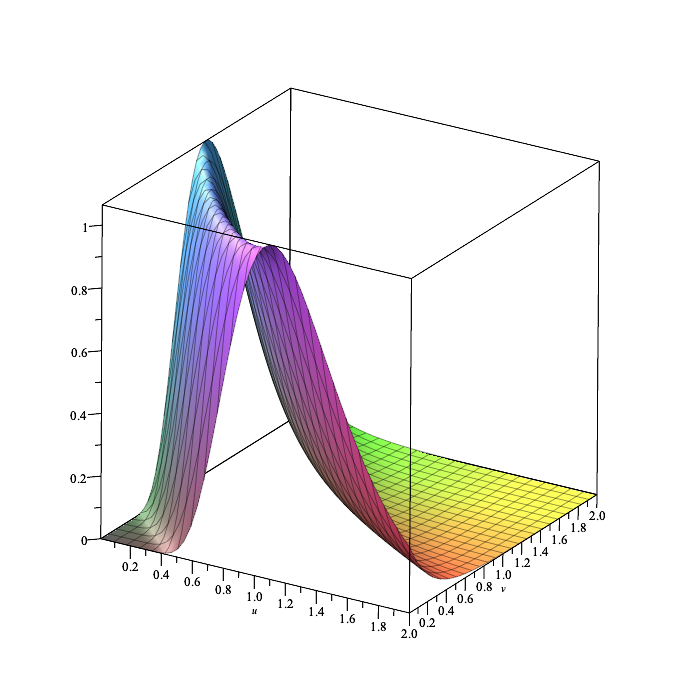}
\includegraphics[scale=0.3]{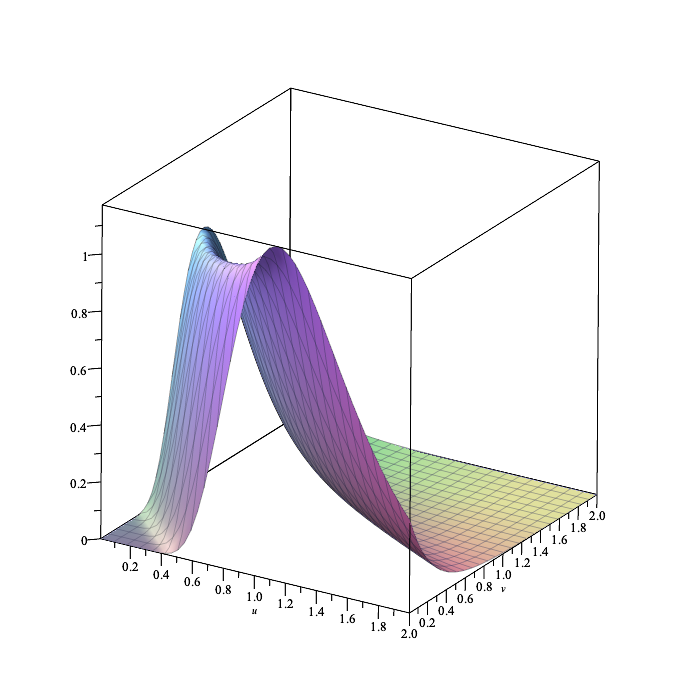}
\vspace{-10mm}
\caption{Example of graph of the density $\pi(u,v)$, on the left for the parameters $(\lambda_1,\lambda_2)=(\frac{1}{2},\frac{1}{2})$, on the right for $(\frac{1}{6},\frac{5}{6})$.}
\label{fig:graph}
\end{figure}
See Figure~\ref{fig:graph} for two illustrations of Theorem~\ref{thm:maindensity}. The compensation method used to show Theorem~\ref{thm:maindensity} is independent of the other techniques developed in our paper; however, to obtain the constants $C$ and $C'$ in \eqref{eq: thm 3}, we make use of Theorem~\ref{thm:explicit_general}.

In the symmetric case, the above result can be simplified as follows:
\begin{corollary}[Density of the invariant measure, symmetric case]
\label{thm:formula_density_original_question}
Introduce three sequences $(a_n)_{n\geq0}$, $(b_n)_{n\geq0}$ and $(c_n)_{n\geq0}$ as follows:\footnote{The first few values of $a_{2n}$ are $10,21,36,55$; those of $b_{2n}$ are $6,15,28,45$; finally, those of $c_{n}$ are $1, -10, 54,- 210, 660, -1782$, see the entry \href{https://oeis.org/A053347}{A053347} in the OEIS (Online Encyclopedia of Integer Sequences).}
\begin{equation*}
\left\{\begin{array}{rcl}
a_{2n} & =& (n+2)(2n+5) ,\\
a_{2n+1} & =& a_{2n},\\
b_{2n} & =&(n+2)(2n+3),\\
b_{2n+1} & =& b_{2n+2},\\
c_n&=&\displaystyle (-1)^n \binom{n+7}{7}\frac{n+4}{4}=(-1)^n\frac{(n + 1)(n + 2)(n + 3)(n + 4)^2(n + 5)(n + 6)(n + 7)}{20160}.\\
 \end{array}\right.
\end{equation*}
If $\lambda_1=\lambda_2=\lambda$, one has
\begin{equation}
\label{eq:formula_density_orthogonal_reflections}
   \pi(u,v)=420\bigl(p(u,v)+p(v,u)\bigr),
\end{equation}
where
\begin{equation}
 \label{eq:p_formula}
   \frac{p(u,v)}{\lambda^2}= \sum_{n\geq0} c_n \exp\bigl(-\lambda(a_{n} u +b_{n}v)\bigr).
\end{equation}
\end{corollary}
Here are the first few terms in the expansion of $p(u,v)$ in \eqref{eq:p_formula}:
\begin{multline*}
   \frac{p(u,v)}{\lambda^2}=e^{-\lambda(10u + 6v)} - 10e^{-\lambda(10u + 15v)} + 54e^{-\lambda(21u + 15v)} - 210e^{-\lambda(21u + 28v)} \\+ 660e^{-\lambda(36u + 28v)} - 1782e^{-\lambda(36u + 45v)}+4290e^{-\lambda(55u + 45v)}-\ldots
\end{multline*}

\noindent {\bf D):} In our last result, we show that in the stationary regime, the distribution of the sum $G+H$ of gaps (in the symmetric case) and the density function ${\nu}_i$ (in the general case) can be written as an infinite convolution of exponential distributions.

\begin{theorem}
\label{thm:sum_exp_sum}
\begin{enumerate}[label={(\roman{*})},ref={(\roman{*})}]
    \item\label{thm4.1}
Assume \eqref{eq:symcond}. Under the stationary distribution ${\bm \pi}$, the probability density function $\,{\bm \sigma} \,$ of the sum $\,G+H\,$ of gaps for the degenerate reflected Brownian motion in \eqref{eq:gapdef}--\eqref{eq:gapSRBM}, is that of the infinite sum $\, \sum_{k=1}^{\infty} {\bm \varepsilon}_{k}\,$ of independent exponential random variables $\, \{ {\bm \varepsilon}_{k}\}_{k \in \mathbb N} \,$ with respective  parameters $\, {\bm \ell}_{k}\,$ given by 
\begin{equation} 
\label{eq: App2}
{\bm \ell}_{k} \, :=\, \frac{\,\lambda\,}{\,8 \,}\bigl((2k+5)^{2} - 1\bigr) \, =\,  \frac{\,\lambda\,}{\,2\,}(k + 2) (k+3), \quad k \in \mathbb N,
\end{equation}
namely, 
\begin{equation} 
\label{eq: App3}
\mathbb P^{\bm \pi} ( G (T) + H (T)\in {\mathrm d} z) \, =\,  {\bm \sigma}(z) {\mathrm d} z \, =\,  \mathbb P^{\bm \pi} \left( \sum_{k=1}^{\infty} {\bm \varepsilon}_{k} \in {\mathrm d} z \right) \, ; \quad z \in [0, \infty) \, .  
\end{equation} 
\item\label{thm4.2} 
The density function ${\nu}_i$ of the measure $\boldnu_i$ is proportional to an infinite convolution of exponential densities with parameters $k(k+\mu_i)(\lambda_1+\lambda_2)$, for $k\in\mathbb Z\setminus \{-1,0,1\}$ for $i =1, 2$.
\end{enumerate}
\end{theorem}


\noindent {\bf Preview:} Our paper is organized as follows. In Section~\ref{sec:bvp}, we introduce various tools and state some preliminary results. We use here an analytical method inspired from \cite{fayolle_random_2017}, initially developed for studying discrete random walks in the quadrant. This method has been recently useful for finding an explicit expression for the Laplace transform of the invariant measure of (non-degenerate)\ reflected Browian motion in the quadrant, see \cite{franceschi_explicit_2017}.  In Section~\ref{sec:Tutte}, we use the invariant approach, developed by William Tutte in the 90's to enumerate colored triangulations
\cite{tutte_chromatic_1995} and applied recently to (non-degenerate)\ reflected Brownian motion in a quadrant in~\cite{franceschi_tuttes_2016,BoMe-El-Fr-Ha-Ra}; this approach allows us to solve a boundary value problem (BVP)\ stated in Section~\ref{sec:bvp}, and thus prove Theorem~\ref{thm:explicit_general} and Corollary~\ref{thm:explicit_symmetric}.
In Section~\ref{sec:boundary_density} we establish Theorem~\ref{cor:boundary_density_non-sym} and Corollary~\ref{cor:boundary_density_sym}, using computations based on the Jacobi-type function $\theta_\mu$ introduced in \eqref{eq:def_theta_mu}.
Theorem~\ref{thm:sum_exp_sum} (together with various extensions)\ is proved in Section~\ref{sec:sum_exp_sum}.
Finally, in Section~\ref{sec:compensation} we prove Theorem~\ref{thm:maindensity} and Corollary~\ref{thm:formula_density_original_question}, using the compensation approach.

\section{Preliminary analytical results}
\label{sec:bvp}

\subsection{Study of the kernel}

We recall the condition \eqref{eq:conddelta} 
and denote by $K$ the kernel
\begin{equation}
\label{def:kernel}
   K(x,y):=(x-y)^2+2(\delta_2-\delta_1)x+2(\delta_3-\delta_2)y
   =(x-y)^2+\lambda_1 x+\lambda_2 y,
\end{equation}
which appears on the left-hand side of the functional equation~\eqref{eq:funceq}. The set of real zeros of the kernel, namely
\begin{equation}
\label{eq:def_parabola_P}
   \mathcal{P}:=\{(x,y)\in\mathbb{R}^2:K(x,y)=0\},
\end{equation}
turns out to be a parabola, see Figure~\ref{fig:parabola}. The straight lines 
\begin{equation*}
2x-y=0
\quad\text{and}\quad
2y-x=0,
\end{equation*}
which appear on the right-hand side of the functional equation~\eqref{eq:funceq},
are also represented on Figure~\ref{fig:parabola}.
\begin{figure}[hbtp]
\centering
\includegraphics[scale=0.55]{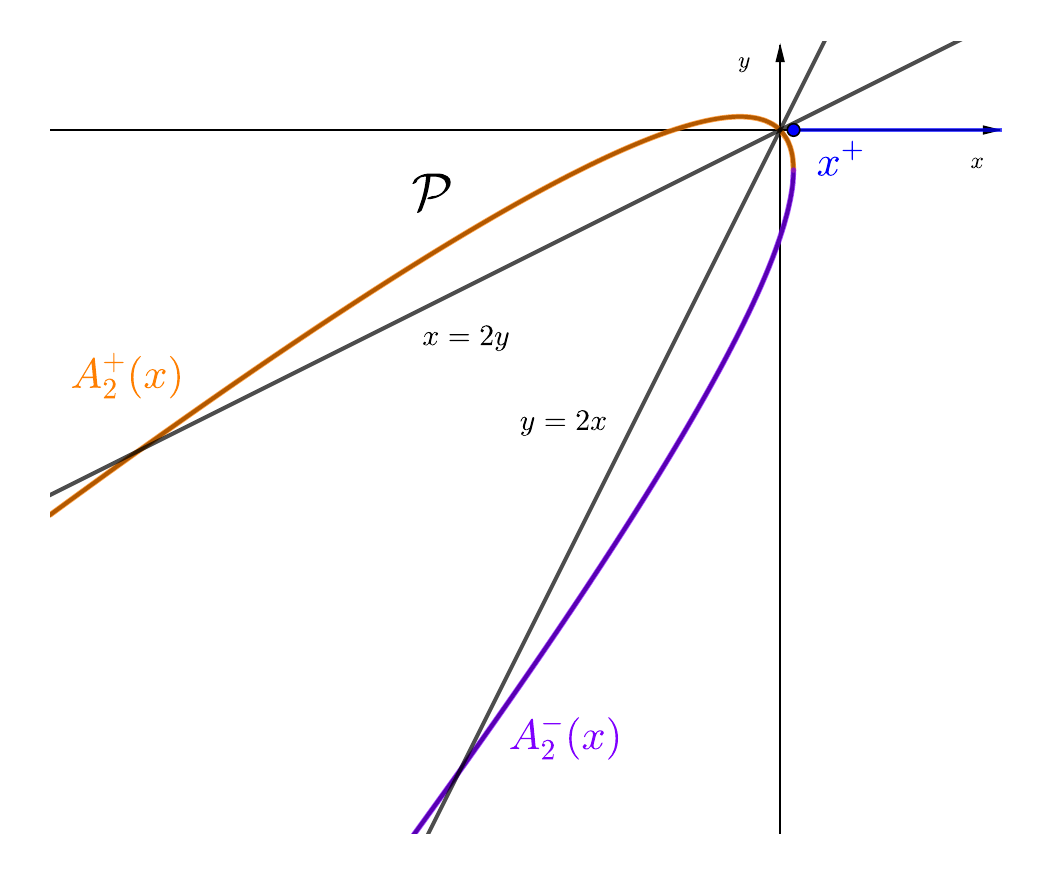}
\caption{The parabola $\mathcal{P}$ and its two branches $A_2^+$ and $A_2^-$, together with the lines of equation $2y-x=0$ and $2x-y=0$.}
\label{fig:parabola}
\end{figure}

We define a bivalued function $A_2$ with two branches
\begin{equation}
\label{A_2}
A_2^\pm(x):=-\frac{\lambda_2}{2}+x\pm\sqrt{\frac{\lambda_2^2}{4}-(\lambda_1+\lambda_2)x}
\end{equation}
which satisfy $K\bigl(x,A_2^\pm(x)\bigr)=0$. Similarly, the functions 
\begin{equation}
\label{A_1}
A_1^\pm(y):=-\frac{\lambda_1}{2}+y\pm\sqrt{\frac{\lambda_1^2}{4}-(\lambda_1+\lambda_2)y}
\end{equation}
satisfy $K\bigl(A_1^\pm(y),y\bigr)=0$.
The branch points of these functions, namely
\begin{equation*}
x^+:=\frac{\lambda_2^2}{4(\lambda_1+\lambda_2)}
\quad\text{and}\quad
y^+:=\frac{\lambda_1^2}{4(\lambda_1+\lambda_2)},
\end{equation*}
cancel the monomials under the square roots in \eqref{A_2} and \eqref{A_1}. 

In the previous definitions, we use the principal determination of the square root. Then, the functions $A_2^+$ and $A_2^-$ (resp.\ $A_1^+$ and $A_1^-$) are defined and analytic on the slit complex plane $\mathbb{C}\setminus [x^+,\infty)$ (resp.\ $\mathbb{C}\setminus [y^+,\infty)$). These branches admit limits on both sides of their respective cut and are complex conjugates on their cut. With a slight abuse of notation, for $y\in [y^+,\infty)$ we shall write $A_1^+(y)=\overline{A_1^-(y)}$; similarly, for $x\in [x^+,\infty)$, we have $A_2^+(x)=\overline{A_2^-(x)}$.

\subsection{Analytic continuation}

In Proposition~\ref{prop:BVP_hatnu_1} below, we will show that the function $\hatnu_1$ of \eqref{eq:def:laplnu1} (and $\hatnu_2$ of \eqref{eq:def:laplnu2})\ satisfies a certain boundary value problem, which will eventually allow us to compute this function explicitly. To that end, we shall need some preliminary results, which we now develop. 

First, we extend $\hatnu_1$ meromorphically from its initial domain of definition (the half-plane with non-negative real parts)\ to a larger domain, 
using the functional equation~\eqref{eq:funceq} together with a standard analytic continuation procedure inspired from~\cite{fayolle_random_2017,franceschi_explicit_2017,BoMe-El-Fr-Ha-Ra}.
\begin{lemma}[Analytic continuation]
\label{lem:analytic_continuation}
The Laplace transform $\hatnu_1$ can be continued analytically  to the open connected set
\begin{equation}
\label{eq:def_S}
S:=\{
y\in \mathbb{C}: 
\Re y\geqslant 0
\text{ or }
\Re A_1^+(y)>0
\}.
\end{equation}
\end{lemma}
\begin{proof}
Notice that similar continuation results have been proved in \cite{franceschi_explicit_2017,BoMe-El-Fr-Ha-Ra}, see in particular \cite[Lem.~3]{franceschi_explicit_2017} and the proof of \cite[Prop.~4.1]{BoMe-El-Fr-Ha-Ra}. We briefly recall here the main details.

Initially, $\hatnu_1$ is defined on the set $\{y\in\mathbb{C}: \Re y \geqslant 0 \}$ and analytic on the interior of this set. Similarly, the Laplace transform $\hatnu_2$ is defined on the set $\{x\in\mathbb{C}: \Re x \geqslant 0 \}$ and is analytic on its interior.
Let us take $y$ in the domain
\begin{equation*}
\{
y\in \mathbb{C}: 
\Re A_1^+(y)>0
\} \cap \{
y\in \mathbb{C}: 
\Re y>0
\}.
\end{equation*}
Observe that this intersection is non-empty (for example, take $y$ real and large enough), see Figure~\ref{fig:domain}. We then evaluate the functional equation \eqref{eq:funceq} at the point $(A_1^+(y),y)$, where the kernel vanishes. We thus obtain the formula
\begin{equation}
\label{eq:continuationformula}
   \hatnu_1(y)=-\frac{2y-{A_1^+(y)}}{2A_1^+(y)-y}\hatnu_2\bigl(A_1^+(y)\bigr),
\end{equation}
which allows us to continue $\hatnu_1$ analytically to the set $\{
y\in \mathbb{C}: 
\Re A_1^+(y)>0
\}$. Indeed, the denominator $A_1^+(y)-\frac{y}{2}$ cannot vanish on this set (since the equation $A_1^+(y)-\frac{y}{2}=0$ has only two solutions, $y=0$ and 
$y=-\lambda_1-2\lambda_2<0$ 
by \eqref{eq:conddelta}). We immediately deduce that $\hatnu_1$ is analytic on the set $S$ in \eqref{eq:def_S}.
\end{proof}

\subsection{An important parabola}

We now introduce a new parabola, namely $\mathcal{P}_2$, which will be used to formulate the BVP in Proposition~\ref{prop:BVP_hatnu_1}:
\begin{equation}
\label{eq:def_P_2}
   \mathcal{P}_2:= A_2^\pm\bigl([x^+,\infty)\bigr)=\bigl\{y\in\mathbb{C}: K(x,y)=0 
\text{ for some }
x\in [x^+,\infty)
 \bigr\},
\end{equation}
see Figure~\ref{fig:domain}.
\begin{lemma}[Parabola $\mathcal{P}_2$]
The curve $\mathcal{P}_2$ in \eqref{eq:def_P_2} is a parabola described by the equation
\begin{equation}
\label{eq:P2}
   \bigl\{y \in \mathbb{C} : (\Im y)^2
   =(\lambda_1+\lambda_2)(\Re y)+\frac{1}{4}\lambda_2(2\lambda_1+\lambda_2)\bigr\}.
\end{equation}
\end{lemma}
\begin{proof}
On the cut $[x^+,\infty)$, the quantities $A_2^\pm(x)$ take complex conjugate values, denoted by $\Re y \pm i \Im y$. By \eqref{A_2}, they satisfy
\begin{equation*}
   \left\{\begin{array}{rcccl}\displaystyle 
A_2^+(x) + A_2^-(x) &=& 2\Re y &=& -\lambda_2+2x,\smallskip
\\
\displaystyle A_2^+(x) A_2^-(x) &=& (\Re y)^2+(\Im y)^2& =& (-\frac{\lambda_2}{2}+x)^2+(\lambda_1+\lambda_2)x-\frac{\lambda_2^2}{4}.
   \end{array}\right.
\end{equation*}
Eliminating $x$ from these two equations readily gives \eqref{eq:P2}.
\end{proof}

\begin{figure}[hbtp]
\centering
\includegraphics[scale=0.7]{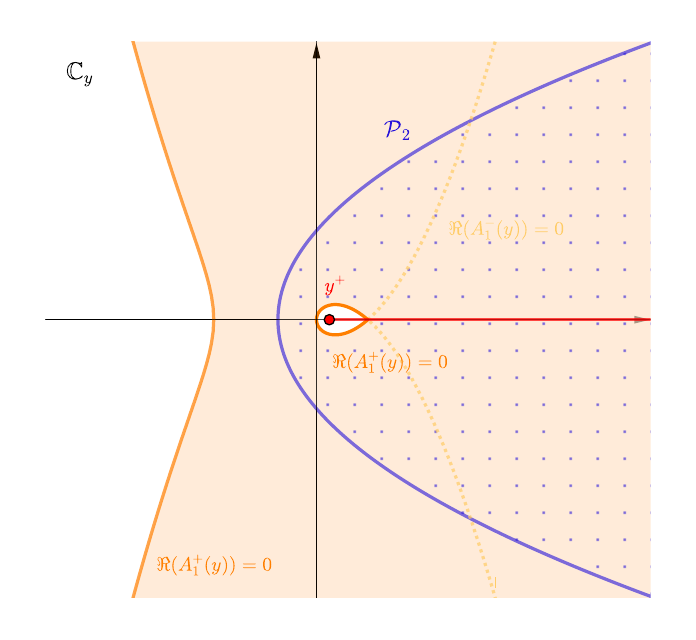}
\caption{The $y$-complex plane; the parabola $\mathcal{P}_2$ is drawn in blue; the domain $\mathcal{D}_2$ inside the parabola is represented by the blue dotted area; the domain $\{
y\in \mathbb{C}: 
\Re A_1^+(y)>0
\}$ is orange and is bounded by the curves of equation $\Re A_1^+(y)=0$; the curve $\Re A_1^-(y)=0$ is also represented by the orange dotted curve.}
\label{fig:domain}
\end{figure}

We denote $\mathcal{D}_2$ the domain inside the parabola $\mathcal{P}_2$, see Figure~\ref{fig:domain}; it is defined by
\begin{equation}
\label{eq:defD2}
   \mathcal{D}_2:= \bigl\{y\in\mathbb{C} : (\Im y)^2<(\lambda_1+\lambda_2)(\Re y)+\frac{1}{4}\lambda_2(2\lambda_1+\lambda_2) \bigr\}.
\end{equation}

\begin{lemma}[Analyticity inside the parabola $\mathcal{P}_2$]
\label{lem:pro2}
The domain $\mathcal{D}_2$ is included in the open, connected set $S$ of~\eqref{eq:def_S}.
As a consequence of Lemma~\ref{lem:analytic_continuation}, the Laplace transform $\hatnu_1$ is analytic in $\mathcal{D}_2$.
\end{lemma}
The inclusion $\mathcal{D}_2\subset S$ can be visualized on Figure~\ref{fig:domain}.

\begin{proof} By the definition of $S$ in \eqref{eq:def_S}, it is obvious that $\mathcal{D}_2 \cap \{y\in\mathbb{C}:\Re y\geqslant 0 \} \subset S$. It remains to show that
\begin{equation*}
   \mathcal{D}_2 \cap \{y\in\mathbb{C}:\Re y < 0 \} \subset \{y\in\mathbb{C}: \Re A_1^+(y)>0 \} \subset S.
\end{equation*}
First, observe that $\mathcal{P}_2 \subset \{y\in\mathbb{C}: \Re A_1^+(y)>0 \}$. Indeed, if $y\in\mathcal{P}_2$, by definition \eqref{eq:def_P_2} there exists $x\in[x^+,\infty)$ such that $K(x,y)=0$ and then $x=A_1^+(y)>0$. For $y$ such that $\Re y=0$, we also have $\Re A_1^+(y)\geqslant 0$ by \eqref{A_1}. Then, $\mathcal{D}_2 \cap \{y\in\mathbb{C}:\Re y\leqslant 0 \}$ is a bounded domain, and for $y$ on its boundary we have $\Re A_1^+(y)\geqslant 0$. The maximum principle applied to $\Re A_1^+$ implies $\Re A_1^+(y)>0 $ for all $y$ in this (open)\ domain. The proof is complete.
\end{proof}

\subsection{Carleman boundary value problem}

Define
\begin{equation}
 \label{eq:defG}
   G(y):=\frac{2A_1^+(y)-y}{2y-A_1^+(y)}\cdot \frac{2\overline{y}-A_1^+(y)}{2A_1^+(y)-\overline{y}}.
\end{equation}
\begin{proposition}[Carleman BVP]
\label{prop:BVP_hatnu_1}
The Laplace transform $\hatnu_1$ satisfies the following Carleman boundary value problem:
\begin{itemize}
   \item $\hatnu_1$ is analytic on the region $\mathcal{D}_2$ of \eqref{eq:defD2};
   \item $\hatnu_1$ satisfies the boundary condition
\begin{equation} 
\label{eq:bvp}
\hatnu_1(\overline{y})=
G(y)\hatnu_1(y), \qquad \forall \ y\in\mathcal{P}_2.
\end{equation}
\end{itemize}
\end{proposition}
The reader may refer to \cite[Sec.~5]{fayolle_random_2017} for a brief summary of the theory of (Carleman)\ boundary value problems; the term ``Carleman'' refers to the fact that a shift function (in our case  complex conjugation)\ is needed to state the BVP.
\begin{proof}
The first point has been seen in Lemma~\ref{lem:pro2}. The second point comes from the functional equation~\eqref{eq:funceq} evaluated at $(A_1^+(y),y)$ and $(A_1^+(y),\overline{y})$ for $y\in\mathcal{P}_2$. Noticing that $A_1^+(y)=A_1^+(\overline{y})\in [x^+,\infty)$ and that at these points the kernel vanishes, we obtain the two equations
\begin{equation*}
\begin{cases}
\bigl(2A_1^+(y)-y\bigr)\hatnu_1(y)+\bigl(2y-A_1^+(y)\bigr)\hatnu_2(A_1^+(y))=0,
\\
\bigl(2A_1^+(y)-\overline{y}\bigr)\hatnu_1(\overline{y})+\bigl(2\overline{y}-A_1^+(y)\bigr)\hatnu_2(A_1^+(y))=0.
\end{cases}
\end{equation*}
We already encountered the first of these equations in \eqref{eq:continuationformula}.
We eliminate now $\hatnu_2(A_1^+(y))$ from these two equations, and obtain the boundary condition~\eqref{eq:bvp}.
\end{proof}

\section{Tutte's invariant approach}
\label{sec:Tutte}

The goal of this section is to provide an explicit solution of the boundary value problem of Proposition~\ref{prop:BVP_hatnu_1} via \emph{Tutte's invariant approach}, which was developed in \cite{tutte_chromatic_1995} and has been applied recently to similar problems in non-degenerate settings; cf.\ the works \cite{BoMe-El-Fr-Ha-Ra,franceschi_tuttes_2016}. 

This method aims to find a \emph{decoupling} function $D$, which is \emph{analytic} (or even rational) in some domain, and such that the function $G$ of \eqref{eq:defG} can be expressed as
\begin{equation}
\label{eq:Gdecoupling}
   G(y)=\frac{D({y})}{D(\overline{y})},  \qquad \forall \ y\in\mathcal{P}_2.
\end{equation}
The above condition is equivalent to 
\begin{equation}
\label{eq:Gdecouplingequiv}
   \frac{\bigl(2x-A_2^-(x)\bigr)\bigl(2A_2^+(x)-x\bigr)}{\bigl(2x-A_2^+(x)\bigr)\bigl(2A_2^-(x)-x\bigr)}=\frac{D(A_2^+(x))}{D(A_2^-(x))}.
\end{equation} 
Such a decoupling function $D$ will be found in Section~\ref{sec:decoupling}. Then, Equation~\eqref{eq:bvp} may be rewritten as
\begin{equation*}
   D(y)\hatnu_1(y)=D(\overline{y})\hatnu_1(\overline{y}), \qquad \forall \ y\in\mathcal{P}_2,
\end{equation*}
which is equivalent to
\begin{equation*}
   D(A_2^+(x))\hatnu_1(A_2^+(x))=D(A_2^-(x))\hatnu_1(A_2^-(x)).
 \end{equation*}
The function $D\hatnu_1$ is then called an \emph{invariant}, because we have for it $(D\hatnu_1)(A_2^+(x))=(D\hatnu_1)(A_2^-(x))$. 

The goal of Tutte's invariant method is to express the \emph{unknown invariant} (here $D\hatnu_1$) in terms of a \emph{canonical invariant}. This is done 
in Section~\ref{sec:tuttegeneral} in the general case. The canonical invariant of our problem is denoted by $W$ and introduced in Section~\ref{sec:conformalgluing}; it happens to be a certain conformal gluing function. 

\subsection{Conformal gluing function}
\label{sec:conformalgluing}

To solve the Carleman boundary value problem of Proposition~\ref{prop:BVP_hatnu_1}, we need to introduce a \emph{canonical} conformal gluing function on the domain $\mathcal{D}_2$, which glues together the upper and the lower parts of the parabola $\mathcal{P}_2$. In the following lemma, the principal determinations of the square root $\sqrt{\cdot}$ and of the logarithm $\ln(\cdot)$ are considered on the slit plane $\mathbb{C}\setminus (-\infty,0]$.

\begin{lemma}[Conformal gluing function]
\label{lem:gluing} 
The function
\begin{equation}
\label{eq:W}
W(y)
=\cosh^2\left(\pi \sqrt{\frac{y}{\lambda_1+\lambda_2}-\frac{\lambda_1^2}{4(\lambda_1+\lambda_2)^2} } \right)
\end{equation}
and its inverse
\begin{equation*}
W^{-1}(z)=\frac{\lambda_1^2}{4(\lambda_1+\lambda_2)}+\frac{\lambda_1+\lambda_2}{\pi^2}\ln^2\left( \sqrt{z} - \sqrt{z-1} \right)
\end{equation*}
satisfy the following properties:
\begin{enumerate}
\item $W$ is conformal (i.e., $W$ is bijective, analytic, and $W^{-1}$ is also analytic) from $\mathcal{D}_2$ to the slit plane $\mathbb{C}\setminus(-\infty,0]$;
\item $W$ glues the parabola $\mathcal{P}$ of \eqref{eq:def_parabola_P} onto $(-\infty,0]$, i.e.,
\begin{equation*}
   W(y)=W(\overline{y})\in (-\infty,0],\qquad \forall \ y\in\mathcal{P},
\end{equation*}
and $W$ is $2$-to-$1$ from $\mathcal{P}\setminus \{A_2^\pm(x^+)\}$ to $(-\infty,0)$ (which means that $W^{-1}$ is bivalued on $(-\infty,0)$).
\end{enumerate}
\end{lemma}
\begin{proof}
Conformal mappings associated with (interior domains of) parabolas are well known in the literature, see for example 
\cite[p.~113]{bieberbach2000conformal} or
\cite[Lem.~5.1]{kapodistria_saxena_boxma_kella_2023}. For $a>0$ define
\begin{equation*}
   z=w(y):=i\cosh\Bigl(\pi\sqrt{ay-\tfrac{1}{4}}\Bigr)
\end{equation*}
and its inverse function
\begin{equation*}
   y=w^{-1}(z):=\frac{1}{4a}+\frac{1}{a\pi^2}\ln^2\Bigl( -iz - \sqrt{-z^2-1} \Bigr).
\end{equation*}
The function $w$ above maps the interior of a parabola to the upper half-plane. More precisely, $w$ is conformal from
\begin{equation*}
   \bigl\{y\in\mathbb{C}: \Re y>a (\Im y)^2 \bigr\}\longrightarrow \bigl\{ z\in\mathbb{C}: \Im z >0 \bigr\}.
\end{equation*}
The equation of the parabola $\mathcal{P}_2$ is $ay^2=x+b$, in accordance with \eqref{eq:P2}, where $a$ and $b$ are chosen as 
\begin{equation*}
\label{eq:def_parameters_a_b}
   a
   :=\frac{1}{\lambda_1+\lambda_2}
\quad\text{and}\quad
b
:=\frac{\lambda_2(2\lambda_1+\lambda_2)}{4(\lambda_1+\lambda_2)}.
\end{equation*}
Then, noticing that $z\mapsto i\sqrt{z}$ maps $\mathbb{C}\setminus(-\infty,0]$ onto the upper half-plane, we define the functions $W^{-1}$ and $W$ by 
\begin{equation*}
   y=W^{-1}(z):=w^{-1}(i\sqrt{z})-b=\frac{1}{4a}-b+\frac{1}{a\pi^2}\ln^2\left( \sqrt{z} - \sqrt{z-1} \right)
\end{equation*}   
and 
\begin{equation*}
   z=W(y):=-w^2(y+b)=\cosh^2\Bigl(
\pi\sqrt{ay+ab-\tfrac{1}{4}}
\Bigr),\end{equation*}
completing the proof.
\end{proof}

\subsection{Decoupling function}
\label{sec:decoupling}

We recall our notation~\eqref{eq:def_lambda_1_2}, as well as the expression of the kernel~\eqref{def:kernel}, namely
\begin{equation*}
   K(x,y)=(x-y)^2+\lambda_1 x+\lambda_2 y.
\end{equation*} 
We introduce the decoupling polynomials 
\begin{equation}
\label{eq:decouplingpol}
   D_1(y):= y(y+\lambda_2)(y+2\lambda_1+\lambda_2
)
\quad\text{and}\quad
D_2(x):= x(x+\lambda_1)(x+2\lambda_2+\lambda_1
).
\end{equation}
This terminology is justified by the following two lemmas.
The identity \eqref{eq:decouplingfactor} right below is crucial: it provides the key step that makes Tutte's invariant method work in our context.
\begin{lemma}[Decoupling identity]
We have
\begin{multline}
\label{eq:decouplingfactor}
\left(x-\frac{y}{2}\right)
\overbrace{y(y+\lambda_2)(y+2\lambda_1+\lambda_2
)}^{D_1(y)}
-
\left(y-\frac{x}{2}\right)
\overbrace{x(x+\lambda_1)(x+2\lambda_2+\lambda_1
)}^{D_2(x)}
\\
=
\frac{1}{2}
\left(
x^2-y^2
+(\lambda_1+2\lambda_2)x
-(2\lambda_1+\lambda_2)y
\right)
\underbrace{\left((x-y)^2+\lambda_1 x+\lambda_2 y\right)}_{K(x,y)}.
\end{multline}
This implies that
\begin{equation}
\label{eq:decoupling3}
\frac{2x-y}{2y-x}
=
\frac{x(x+\lambda_1)(x+2\lambda_2+\lambda_1
)}{y(y+\lambda_2)(y+2\lambda_1+\lambda_2
)} =\frac{D_2(x)}{D_1(y)},
\quad\text{as long as } K(x,y)=0.
\end{equation}
\end{lemma}

\begin{lemma}[Decoupling lemma] 
\label{lem:decoupling}
The following \emph{decoupling} relation holds:
\begin{equation}
\label{eq:decoupling1}
   \frac{\bigl(2x-A_2^-(x)\bigr)\bigl(2A_2^+(x)-x\bigr)}{\bigl(2x-A_2^+(x)\bigr)\bigl(2A_2^-(x)-x \bigr)}
=
 \frac{A_2^+(x)\bigl(A_2^+(x)+\lambda_2\bigr)\bigl(A_2^+(x)+2\lambda_1+\lambda_2\bigr)}{A_2^-(x)\bigl(A_2^-(x)+\lambda_2\bigr)\bigl(A_2^-(x)+2\lambda_1+\lambda_2\bigr)}.
\end{equation}
Accordingly, the function $G$ defined in \eqref{eq:defG} may be written in the form \eqref{eq:Gdecoupling}  with $D=1/D_1$; namely, as
\begin{equation}
\label{eq:decoupling2}
   G(y)=\frac{\overline{y}(\overline{y}+\lambda_2)(\overline{y}+2\lambda_1+\lambda_2)}{y(y+\lambda_2)(y+2\lambda_1+\lambda_2)}=\frac{D_1(\overline{y})}{D_1(y)}, \qquad \forall y\in\mathcal{P}_2.
\end{equation}
\end{lemma}
\begin{proof}
Evaluating \eqref{eq:decoupling3} at $(x,A_2^+(x))$, then at $(x,A_2^-(x))$, then dividing, we arrive at \eqref{eq:decoupling1}. Recalling~\eqref{eq:Gdecouplingequiv}, and noticing that for $y\in\mathcal{D}_2$ we have $y=A_2^-(A_1^+(y))$ and $\overline{y}=A_2^+(A_1^+(y))$, we re-cast~\eqref{eq:decoupling1}~in the form~\eqref{eq:Gdecoupling}, \eqref{eq:Gdecouplingequiv} with the function $G$ given as in~\eqref{eq:decoupling2}, by continuity on $\mathcal{P}_2$, the boundary of the domain $\mathcal{D}_2$ in~\eqref{eq:defD2}.
\end{proof}

\subsection{Explicit expression of the Laplace transform}
\label{sec:tuttegeneral}

Thanks to the decoupling Lemma~\ref{lem:decoupling} and the Carleman boundary value problem of Proposition~\ref{prop:BVP_hatnu_1}, we obtain a new invariant relationship for $\hatnu_1$. 

\begin{lemma}[Invariance]
\label{lem:invariant2}
The Laplace transform $\hatnu_1$ satisfies the following \emph{invariance} relation on the parabola $\mathcal{P}_2$:
\begin{equation*}
   \frac{\hatnu_1(\overline{y})}{\overline{y}(\overline{y}+\lambda_2)(\overline{y}+2\lambda_1+\lambda_2)}=
   \frac{\hatnu_1(y)}{y(y+\lambda_2)(y+2\lambda_1+\lambda_2)}, \qquad\forall \ y\in\mathcal{P}_2.
\end{equation*}
\end{lemma}
\begin{proof}
This follows directly from Proposition~\ref{prop:BVP_hatnu_1} and Lemma~\ref{lem:decoupling}.
\end{proof}

\begin{proof}[Proof of Theorem~\ref{thm:explicit_general} and Corollary~\ref{thm:explicit_symmetric}]
The key point of Tutte's invariant method consists in expressing the invariant of Lemma~\ref{lem:invariant2} in terms of the canonical conformal gluing function $W$ studied in Section~\ref{sec:conformalgluing}. Let us denote
\begin{equation} \label{eq: def f} 
   f(y)=\frac{\hatnu_1({y})}{y(y+\lambda_2)(y+2\lambda_1+\lambda_2)}-\frac{\hatnu_1(0)}{\lambda_2(2\lambda_1+\lambda_2)}\frac{W'(0)}{W(y)-W(0)},
\end{equation}
and remark that the function $f$ has no pole at $0$: by construction, the residue of $f$ at $0$ is zero. Furthermore, we observe that the points $-\lambda_2$ and $-2\lambda_1-\lambda_2$ are not in $\mathcal{D}_2$. 

We want to show that $f\equiv 0$. Lemma~\ref{lem:gluing} and Lemma~\ref{lem:invariant2} imply that $f$ satisfies the following boundary value problem:
\begin{itemize}
   \item $f$ is analytic in $\mathcal{D}_2$ and continuous on its boundary $\mathcal{P}_2$;
   \item $f$ satisfies the boundary condition $f(\overline{y})=f(y)$, for all $y\in\mathcal{P}_2$;
   \item $f(y) \rightarrow 0$ when $\vert y\vert  \to \infty$.
\end{itemize}
These three properties together imply that $f\equiv 0$, as can be deduced from Lem.~2 in \cite[Sec.~10.2]{litvinchuk_solvability_2000}. 

To provide some more concrete arguments leading to the conclusion $f\equiv 0$, we notice that $f\circ W^{-1}$ is analytic on the whole of $\mathbb{C}$ 
and goes to $0$ at infinity, and is therefore equal to $0$ using the classical Liouville theorem. 

Thanks to the functional equation~\eqref{eq:funceq} and using $\hatpi(0,0)=1$ since $\boldpi$ is a probability measure, we can show that
\begin{equation}
\label{eq:value_hatnu_1(0)}
    \hatnu_1(0)=\frac{2}{3}(2\lambda_1+\lambda_2).
\end{equation}
Using formula~\eqref{eq:W}, the properties of the cosine, 
as well as 
\begin{equation*}
W(0)=\cosh^2 \Bigl( i\pi \frac{\lambda_1}{2(\lambda_1+\lambda_2)} \Bigr), 
\end{equation*}
we obtain
\begin{align*}
W(y)-W(0)&=\cos^2\left(\pi \sqrt{\frac{\lambda_1^2}{4(\lambda_1+\lambda_2)^2}-\frac{y}{\lambda_1+\lambda_2}} \right)-\cos^2 \left( \pi \frac{\lambda_1}{2(\lambda_1+\lambda_2)} \right)
\\ &=
\frac{1}{2}
\cos \left( \pi \sqrt{\frac{\lambda_1^2}{(\lambda_1+\lambda_2)^2}-\frac{4y}{\lambda_1+\lambda_2}} \right)
-\frac{1}{2}\cos \left( \pi \frac{\lambda_1}{\lambda_1+\lambda_2}   \right) ,
\end{align*}
and compute \begin{equation*}
W'(0)=\frac{\pi}{\lambda_1} \sin \Bigl(\frac{\pi \lambda_1}{\lambda_1 + \lambda_2} \Bigr). 
\end{equation*}
Substituting these last three computations into \eqref{eq: def f},  and recalling $f \equiv 0$, we see that the proof of Theorem~\ref{thm:explicit_general} is now complete. Corollary~\ref{thm:explicit_symmetric} follows as an immediate consequence, upon taking $\lambda=\lambda_1=\lambda_2$.
\end{proof}

\section{Boundary densities}
\label{sec:boundary_density}


\subsection{Relation with the bivariate density}

For further use, it is convenient to interpret the densities $\nu_1(v)$ and $\nu_2(u)$ of the ``lateral measures'' in \eqref{eq:def:nu1} as the specialisations of the bivariate density $\pi(u,v)$ at $u=0$ and $v=0$, respectively. This will be used in particular to compute the constants $C$ and $C'$ appearing in Theorem~\ref{thm:maindensity}.
\begin{proposition}[Specialisations of $\pi$] We have $\pi(u,0)=\nu_2(u)$ and $\pi(0,v)=\nu_1(v)$.
\label{prop:boundarydensity}
\end{proposition}
\begin{proof}
The initial value formula gives
\begin{equation*}
   \lim_{x\to\infty} x \hatpi(x,y)=\int_{\mathbb{R}_+} e^{-yv} \pi(0,v)\mathrm{d}v.
\end{equation*}
By dividing the functional equation~\eqref{eq:funceq} by $x$ and letting $x\to\infty$, we obtain
\begin{equation*}
   \lim_{x\to\infty} x \hatpi(x,y)=\hatnu_1(y)=\int_{\mathbb{R}_+} e^{-yv} \nu_1(v)\mathrm{d}v.
\end{equation*}
Comparing the two limits, we conclude that $\pi(0,v)=\nu_1(v)$. A similar argument shows that $\pi(u,0)=\nu_2(u)$.
\end{proof}

\subsection{Proof of Theorem~\ref{cor:boundary_density_non-sym} via Mittag-Leffler expansions}

We summon now from Theorem~\ref{thm:explicit_general}, in order to provide a proof of Theorem~\ref{cor:boundary_density_non-sym}.
Recall the Jacobi theta-type function 
\begin{equation*}
   \theta_{\mu_1}(q) = \sum_{n\in\mathbb Z} \Bigl(n+\frac{\mu_1}{2}\Bigr) q^{n(n+\mu_1)}, \quad q\in (0,1)
\end{equation*} 
introduced in \eqref{eq:def_theta_mu},
with $\mu_1=\frac{\lambda_1}{\lambda_1+\lambda_2}$ defined in \eqref{eq:def_mu}. The reason for introducing this function 
appears in the following important technical result, which shows that $\theta_{\mu_1}$ is naturally and intrinsically connected to the Laplace transform of the lateral measures $\boldnu_1$ in \eqref{eq:expression_hatnu_1_nonsym}, and thus also to its density function ${ \nu}_1  $ as in \eqref{eq: nu1(u)}. 
\begin{lemma}[Laplace transform of $\theta_\mu$]
\label{lem:Lap_theta_mu}
The Laplace transform of the function $u\mapsto \theta_{\mu_1} \left(e^{-(\lambda_1+\lambda_2) u} \right)$ 
is given for any $x\geq0$, by  
\begin{equation*}
    \int_0^\infty \theta_{\mu_1}(e^{-(\lambda_1+\lambda_2)u})e^{-ux}\mathrm{d}u = \frac{1}{\lambda_1+\lambda_2}\frac{\pi\sin(\pi\mu_1)}{
   \cos \Bigl( \pi \sqrt{\mu_1^2-\frac{4x}{\lambda_1+\lambda_2}} \Bigr)
   -\cos \Bigl( \pi \mu_1   \Bigr)}.
\end{equation*}    
\end{lemma}

Before proving Lemma~\ref{lem:Lap_theta_mu}, we first show how it implies Theorem~\ref{cor:boundary_density_non-sym}.
Let us recall the decoupling polynomial $D_1$ introduced in~\eqref{eq:decouplingpol}, and the differential operator $\D_1$ introduced in~\eqref{eq:diff_op} for a smooth function $f$ as 
\begin{equation*} 
\label{eq:defD1f}
\D_1[f] := D_1\left( \frac{\mathrm{d}}{\mathrm{d}u} \right) [f] = f'''+2(\lambda_1 + \lambda_2) f''+\lambda_2(\lambda_2 + 2\lambda_1) f'.
\end{equation*} 
We also introduce its dual operator
\begin{equation*} 
\label{eq:defD1fdual} \D_1^*[f] := -f'''+2(\lambda_1 + \lambda_2) f''-\lambda_2(\lambda_2 + 2\lambda_1) f'.
\end{equation*} 
Using  Theorem~\ref{thm:explicit_general} and Lemma~\ref{lem:Lap_theta_mu} together, we write the expression \eqref{eq:expression_hatnu_1_nonsym} as  
\begin{align}
\label{eq:before_IBP}
    \hatnu_1(y)
    &=\frac{4(\lambda_1+\lambda_2)}{3\lambda_1\lambda_2} D_1(y) \int_0^\infty \theta_{\mu_1} \left(e^{-(\lambda_1+\lambda_2)u}\right)e^{-uy}\mathrm{d}u
   \\ &= 
    \frac{4(\lambda_1+\lambda_2)}{3\lambda_1\lambda_2}  \int_0^\infty \theta_{\mu_1}\left(e^{-(\lambda_1+\lambda_2)u}\right) \D_1^*[e^{- \boldsymbol{\cdot} y }](u)\mathrm{d}u\nonumber
    \\ &= 
    \frac{4(\lambda_1+\lambda_2)}{3\lambda_1\lambda_2}  \int_0^\infty \D_1 \left[ \theta_{\mu_1}\left(e^{-(\lambda_1+\lambda_2) \boldsymbol{\cdot} }\right) \right](u) e^{-uy} \mathrm{d}u
    .
    \label{eq:after_IBP}
\end{align}
To prove the last equality, we simply integrate by parts in \eqref{eq:before_IBP} on the dual operator. In that step, we crucially use the fact that $\theta_{\mu_1}$ and all its derivatives tend to $0$ as $q \to 1$, $q<1$. While this property is not clear at all from the definition \eqref{eq:def_theta_mu} of $\theta_{\mu_1}$, it appears as a direct consequence of the crucial Lemma~\ref{lem:Jacobi_transform} below, which establishes a Jacobi-type modular identity for $\theta_{\mu_1}$.

Thanks to Equation~\eqref{eq:after_IBP} and using classical results on the injectivity of the Laplace transform, we deduce that
\begin{equation*}
   \nu_1(u)= \frac{4(\lambda_1+\lambda_2)}{3\lambda_1\lambda_2}  \D_1 \left[ \theta_{\mu_1}\left(e^{-(\lambda_1+\lambda_2) \boldsymbol{\cdot} }\right) \right](u) \,, 
\end{equation*}
that is, the claim \eqref{eq: nu1(u)}. After some computations and simplifications, one easily finds
\begin{multline*}
    \D_1 \left[ \theta_{\mu_1}\left(e^{-(\lambda_1+\lambda_2) \boldsymbol{\cdot} }\right) \right](u)
   = \sum_{n\in\mathbb Z} (n+{\mu_1}/{2})
   D_1\Bigl(-(\lambda_1+\lambda_2)n(n+\mu_1)\Bigr)
   e^{-(\lambda_1+\lambda_2) u }
    = -(\lambda_1 + \lambda_2)^3\times\\\sum_{n\in\mathbb Z}(n-1)n(n+1)(n - 1 + \mu_1)
    (n + \mu_1)(n + 1 + \mu_1 )(n + \mu_1 /2)e^{-(n^2+\mu_1 n)(\lambda_1+\lambda_2)u}.
\end{multline*}
We thus have proved Theorem~\ref{cor:boundary_density_non-sym}. What remains is to prove Lemma~\ref{lem:Lap_theta_mu}.

\begin{proof}[Proof of Lemma~\ref{lem:Lap_theta_mu}]
Integrating term by term, one finds
\begin{align*}
    \int_0^\infty \theta_{\mu_1}(e^{-(\lambda_1+\lambda_2)u})e^{-ux}d\mathrm{d}u &= \sum_{n\in\mathbb Z} \frac{n+\frac{\mu_1}{2}}{x+(\lambda_1+\lambda_2)n^2+\lambda_1 n} \\&= -\frac{2}{\mu_1(\lambda_1+\lambda_2)}\sum_{n\in\mathbb Z} \frac{1+\frac{2}{\mu_1}n}{s^2-\bigl(1+\frac{2}{\mu_1}n\bigr)^2},
\end{align*}
where we have set $s^2=1-\frac{4}{\mu_1^2}\frac{x}{\lambda_1+\lambda_2}$. The right-hand side of the above identity is then computed from the following classical result, the proof of which is omitted:
\begin{lemma}[Mittag-Leffler expansion of shifted cosine]
Let $\mu\in(0,1)$. One has for $s\in\mathbb C$
\begin{equation*}
   \frac{\pi}{\cos(\pi\mu s) - \cos(\pi\mu)} =  -\frac{2}{\mu\sin(\pi\mu)}\sum_{n\in\mathbb Z}\frac{1 + \frac{2}{\mu}n}{s^2 - \bigl(1 + \frac{2}{\mu}n\bigr)^2}.
\end{equation*}
\end{lemma}
The proof of Lemma~\ref{lem:Lap_theta_mu} is complete.
\end{proof}

\begin{lemma}[Jacobi transformation for $\theta_\mu$]
\label{lem:Jacobi_transform}
For any $u\geq 0$,
\begin{equation*}
 \theta_\mu(e^{-u}) =  \sum_{n\in\mathbb Z}\Bigl(n + \frac{\mu}{2}\Bigr)\exp\bigl(-n(n + \mu)u\bigr) = \frac{\pi^{3/2}}{u^{3/2}}
\sum_{n\in\mathbb Z} n\sin(\pi\mu n)\exp\Bigl(-\frac{\pi^2 n^2}{u} + \frac{\mu^2 u}{4}\Bigr).
\end{equation*}    
\end{lemma}
\begin{proof}
    Consider the function $f(x)=(x+\frac{\mu}{2})e^{-ux(x+\mu)}$ and its Fourier transform
    $\widehat f (y)=\int_{-\infty}^{\infty}f(x)e^{-2i\pi yx} \mathrm{d}x$. 
   The classical Poisson summation formula expresses now the function of \eqref{eq:def_theta_mu}   as 
  \begin{equation*}
   \theta_\mu(e^{-u})
  =\sum_{n\in\mathbb Z} f(n) 
  = \sum_{n\in\mathbb Z} \widehat f(n) , 
  \end{equation*}
where a direct computation gives
  \begin{align*}
     \widehat f (y)  &=
     \int_{-\infty}^{\infty} \left(x+\frac{\mu}{2}\right)e^{-ux(x+\mu)-2i\pi yx} \mathrm{d}x
     \\ &= e^{\frac{\mu^2 u}{4} +i\pi y \mu}
     \int_{-\infty}^{\infty} \left(x+\frac{\mu}{2}\right)e^{-u(x+\frac{\mu}{2})^2-2i\pi y (x+\frac{\mu}{2}) } \mathrm{d}x
     \\ &=
     e^{\frac{\mu^2 u}{4} +i\pi y \mu}
   \frac{(-1)}{2i\pi} \frac{\mathrm{d}}{\mathrm{d}y} 
   \underbrace{
     \int_{-\infty}^{\infty} e^{-u(x+\frac{\mu}{2})^2-2i\pi y (x+\frac{\mu}{2}) } \mathrm{d}x
}_{\sqrt{\frac{\pi}{u}} e^{-\frac{\pi^2 y^2}{u}}}
\\ &=
e^{\frac{\mu^2 u}{4} +i\pi y \mu}
   \frac{1}{2i\pi} \sqrt{\frac{\pi}{u}}  \frac{\pi^2}{u} 2y e^{-\frac{\pi^2 y^2}{u}} 
   \\ &= y\frac{\pi^{3/2}}{u^{3/2}} \bigl(\sin(\pi\mu y) - i \cos (\pi\mu y)\bigr) \exp\Bigl(-\frac{\pi^2 y^2}{u} + \frac{\mu^2 u}{4}\Bigr)  .
  \end{align*}
  The result of Lemma~\ref{lem:Jacobi_transform} follows directly from this last expression and the Poisson summation formula, noting that the imaginary parts disappear for parity reasons, when summing over $\mathbb{Z}$. 
\end{proof}

\begin{remark}
    The Jacobi theta-like function $\theta_\mu$ of \eqref{eq:def_theta_mu} is positive for $q\in[0,1)$. This is not fully clear when looking at the definition \eqref{eq:def_theta_mu}, but  follows easily from the probabilistic interpretation of $\theta_\mu$ provided in Appendix~\ref{sec:app_theta_mu}.
\end{remark}

Let us mention the paper \cite{Salminen-Vignat-23} by Salminen and Vignat, which interprets the four modular identities for ``classical'' theta functions in terms of Brownian motion either reflected or killed at the endpoints of an interval. Here, we introduce a ``novel'' theta-like function $\theta_\mu$ as in \eqref{eq:def_theta_mu}, derive for it the modular identity of Lemma \ref{lem:Jacobi_transform}, and connect it with Brownian motion conditioned to live forever inside a given interval; see Appendix~\ref{sec:app_theta_mu} for this interpretation and connection.

\subsection{Further comments on the symmetric case of Corollary~\ref{cor:boundary_density_sym}}

Specializing Theorem~\ref{cor:boundary_density_non-sym} to the symmetric case  $\lambda=\lambda_1=\lambda_2=\lambda$ (thus $\mu_i=\frac{1}{2}$ in~\eqref{eq:def_mu}), one finds
\begin{multline*}
    \nu_2(u) =\\ -\frac{64}{3}\lambda^2\sum_{n\in\mathbb Z}(n-1)n(n+1)(n - 1/2)
    (n + 1/2)(n + 3/2)(n + 1/4)\exp(-n(2 n+1)\lambda u).
\end{multline*}
Looking at the exponents appearing in the above exponential terms, a simplification occurs due to the fact that
\begin{equation}
    \label{eq:equality_orbits}
    \bigl\{n(2 n+1): n\in\mathbb Z\bigr\}=\bigl\{\tfrac{1}{2}{n(n+1)} : n\geq 0\bigr\}.
\end{equation}
More precisely, the $\frac{n(n+1)}{2}$ for even (resp.\ odd) values of $n\geq 0$ correspond to the $n(2 n+1)$ for non-negative (resp.\ negative) values of $n$. Performing a straightforward change of index, one gets
\begin{align*}
    \nu_2&(u) =\\ & \frac{64 \lambda^2}{3\times 256}\sum_{n\geq 0} (-1)^{n-1}\frac{(n-2)(n-1)n(n+1)(n+2)(n+3)(2n + 1)}{256}\exp\Bigl(-\frac{n(n+1)}{2}\lambda u\Bigr),
\end{align*}
which proves Corollary~\ref{cor:boundary_density_sym} since $3\times 256/64=12$.


\begin{remark}[Relation with the Jacobi theta function]
    Although not crucial for our purpose, it is interesting to observe that the function $\theta_\mu$ in \eqref{eq:def_theta_mu} simplifies to a classical Jacobi theta function (sometimes called Jacobi constant)\ in the symmetric case $\mu=\frac{1}{2}$:
    \begin{equation*}
        4\theta_{\frac{1}{2}}(q^4) = \theta_{1,1}(q)= \sum_{n=-\infty}^\infty (-1)^n nq^{n(n+1)}.
    \end{equation*}  
This function, and its generalizations
\begin{equation*}
   \theta_{m,k}(q) = \sum_{n=-\infty}^\infty (-1)^n n(n-1)\cdots (n-k+1)q^{n^2+mn},
\end{equation*}
are studied by Huber in \cite{Huber-08}.
\end{remark}

\section{Sum of gaps as an infinite sum of independent exponential variables}
\label{sec:sum_exp_sum}

\subsection{The symmetric case}

Let us recall the notation of \eqref{eq:gapdef}--\eqref{eq:recurrence_condition} for the degenerate reflected Brownian motion $\, (G(\cdot), H(\cdot))\,$. The basic adjoint relation of \eqref{eq:funceq} describes the stationary distribution $\,{\bm \pi} \,$ of this process. We prove the first part of Theorem~\ref{thm:sum_exp_sum}.

\begin{proof}[Proof of Theorem~\ref{thm:sum_exp_sum} \ref{thm4.1}]
When $\, \lambda_{1}\, =\,  \lambda_{2}\, =\,  \lambda\,$, the Laplace transform \eqref{eq:expression_hatnu_1_sym} in Corollary~\ref{thm:explicit_symmetric} has the product form 
\begin{equation} \label{eq: App1}
\mathbb E ^{\bm \pi} [ e^{-y(G(T)+H(T))}] \, =\, \widehat{\bm \pi} (y, y) \, =\, \frac{\,\widehat{\bm \nu}_{1} (y)\,}{\,2\lambda\,} \, =\,  \frac{\,\widehat{\bm \nu}_{2}(y)\,}{\,2\lambda\,} \, =\,  \prod_{k=1}^{\infty} \frac{\,{\bm \ell}_{k}\,}{\, y + {\bm \ell}_{k}\,}  \, ; \quad y \in [0, \infty) 
\end{equation}
of those of exponential random variables with parameters \eqref{eq: App2}.

To verify \eqref{eq: App1}, first let us apply the infinite product formula 
\begin{equation}
\label{eq:cos_infinite_product_formula}
\cos (\sqrt{-1} z) \, =\, \cosh (z) \, =\,  \prod_{k=1}^{\infty} \Big( 1 + \frac{\,4 z^{2}\,}{\,\pi^{2}(2k-1)^{2}\,} \Big)\, ; \quad z \in \mathbb  C \, 
\end{equation}
of the hyperbolic cosine function to the expression \eqref{eq:expression_hatnu_1_sym}, and obtain for $\, y \in [0, \infty) \,$ 
\begin{equation*}
\frac{\,\hatnu_1(y)\,}{\,2 \lambda\,} \, =\,  \frac{\,2 \pi y (y+\lambda) ( y + 3 \lambda)\,}{\,3 \lambda^{3} \cos ( (\pi/2) \sqrt{ 1 - (8y/\lambda)}) \,} \, =\,  \frac{\,2\pi\,}{\,3 \lambda^{3}\,}y(y+\lambda) (y+3 \lambda) \prod_{k=1}^{\infty} \Big( 1 + \frac{\,(8y/\lambda) - 1\,}{\,  (2k-1)^{2}\,}\Big)^{-1} \, . 
\end{equation*}
Note that the first three terms  $\, \lambda / 8 y  \,$, $\, 9\lambda/ (8 (y+\lambda))\,$ and $\, 25 \lambda /(8 (y + 3 \lambda))\,$ with $\, k \, =\,  1, 2, 3\,$ in the infinite product  may be cancelled with the constant multiple of $\, y (y+\lambda)(y+3\lambda)\,$. Also, note that after the cancellation, the term in the infinite product is rewritten as 
\begin{equation*}
\Big( 1 + \frac{\,(8y/\lambda) - 1\,}{\,  (2k-1)^{2}\,}\Big)^{-1} \, =\,  \frac{\,(2k-1)^{2}\,}{\,(8y/\lambda) + (2k-1)^{2}-1 \,} \, =\, \Big( 1 + \frac{\,1\,}{\,4 k (k-1)\,} \Big) \cdot \frac{\,{\bm \ell}_{k-3}\,}{\,y + {\bm \ell}_{k-3}\,}  \, ; \quad k \ge 4 \, ,   
\end{equation*}
where $\, {\bm \ell}_{k}\,$ is defined in \eqref{eq: App2} for $\, k \in \mathbb N\,$. Thus, with these considerations, we obtain 
\begin{equation*}
\frac{\,\hatnu_1(y)\,}{\,2 \lambda\,} \, =\,  \frac{\,2 \pi\,}{\,3\,} \cdot \frac{\,9 \cdot 25\,}{\,8^{3}\,} \cdot \prod_{k=4}^{\infty} \Big( 1 + \frac{\,1\,}{\,4 k (k-1)\,} \Big) \cdot \prod_{k=4}^{\infty} \frac{\,{\bm \ell}_{k-3}\,}{\,y + {\bm \ell}_{k-3}}\,  =\,  \prod_{k=1}^{\infty} \frac{\,{\bm \ell}_{k}\,}{\,y + {\bm \ell}_{k}\,} \, , 
\end{equation*}
i.e., \eqref{eq: App1}, because of the infinite product $\,\prod_{k=4}^{\infty} ( 1 + \frac{\,1\,}{\,4 k (k-1)\,} ) \, =\, \frac{256}{75 \pi}\,$. 

\smallskip 

Each exponential random variable $\, {\bm \varepsilon}_{k}\,$ with parameter $\, {\bm \ell}_{k}\,$ has expectation $\, \mathbb E^{\bm \pi} [ {\bm \varepsilon}_{k} ] \, =\,  1 / {\bm \ell}_{k}\,$, $\, k \in \mathbb N\,$, and hence, the infinite series $\, \sum_{k=1}^{\infty} {\bm \varepsilon}_{k}\,$ has the (finite)\ expectation 
\begin{equation*}
 \mathbb E  ^{\bm \pi} \sum_{k=1}^{\infty} {\bm \varepsilon}_{k}  \, =\, \sum_{k=1}^{\infty} \frac{\,1\,}{\,{\bm \ell}_{k}\, } \, =\,  \frac{\,2\,}{\,\lambda\,} \sum_{k=1}^{\infty} \frac{\,1\,}{\,(k+2)(k+3)\,} \, =\,  \frac{\,2\,}{\,3\lambda\,}  \, , 
\end{equation*}
in accordance with (2.76) of \cite{ichiba_karatzas_degenerate_22-arxiv}. This implies $\, \sum_{k=1}^{\infty} {\bm \varepsilon}_{k} < \infty\,$ almost surely, and hence, the corresponding stationary density $\, {\bm \sigma} \,$ in \eqref{eq: App3} is well defined. 
\end{proof}

The stationary distribution of the sum $\, G(T) + H(T)\,$ is given by the {\it infinite} convolution of exponential distributions with parameters $\, \{ {\bm \ell}_{k}\}_{k \in \mathbb N}\,$. It is infinitely divisible with  L\'evy density $\sum_{k=1}^\infty e^{-{\bm \ell}_k z }/z $, $z > 0$, that is, 
\begin{equation*}
\mathbb E^{\bm \pi} [ e^{-y(G(T)+H(T))}] = \exp \Big( - \int^\infty_0 ( 1 - e^{-yz} ) \, \frac{1}{z}\sum_{k=1}^\infty e^{-{\bm \ell}_k z }  \, \mathrm d z   \Big) \, ; \quad y \in [0, \infty) \,.
\end{equation*}
The convolution of finitely many exponential distributions is known to be the Coxian (or ``hypoexponential''  or ``phase-type'') distribution, and is used in queueing theory. The connection between infinite sums of exponential random variables and infinitely divisible distributions is discussed in \cite{biane_pitman_yor_01} and \cite{kapodistria_saxena_boxma_kella_2023}. 

Moreover, the stationary density function ${\bm \sigma}$ of $G(T) + H(T)$ provides the marginal stationary density function of $G(T)$. It follows from \eqref{eq: App3}  that it has an exponential moment
\begin{equation} \label{eq: App4}
\begin{split}
&  \int^{\infty}_{0} e^{\lambda z } {\bm \sigma}(z) {\mathrm d} z \, =\, \mathbb E^{\bm \pi} [ e^{\lambda (G(T)+H(T))} ] \, =\,   \mathbb E^{\pi} [ e^{\lambda \sum_{k=1}^{\infty} {\bm \varepsilon}_{k}} ]  \\
\, & \quad  =\,  \mathbb E^{\bm \pi}[ \prod_{k=1}^{\infty} e^{\lambda {\bm \varepsilon}_{k}} ] \, =\, \prod_{k=1}^{\infty} \mathbb E^{\bm \pi} [ e^{\lambda {\bm \varepsilon}_{k}}] 
 =\,  \prod_{k=1}^{\infty} \frac{\,{\bm \ell}_{k}\,}{\,{\bm \ell}_{k}- \lambda\,} \\
 & \quad =\,  \prod_{k=1}^{\infty} \Big( 1 + \frac{\,\lambda\,}{\,{\bm \ell}_{k} - \lambda\,} \Big) \, =\,  \prod_{k=1}^{\infty} \Big( 1 + \frac{\,1\,}{\,(k(k+5)/2) + 1 \,} \Big) \, =\,  2 \, . 
\end{split}
\end{equation}
Observe now that the Laplace transform $\, \widehat{\bm \pi} (x, 0) \,$  of the marginal stationary distribution of $\, G(T)\,$ (and hence of $\,H(T)\,$, because of the symmetry) is determined from \eqref{eq:funceq} by the Laplace transform $\, \widehat{\bm \pi} (x, x) \,$ along the diagonal, i.e., 
\begin{equation*}
\mathbb E ^{\bm \pi} [ e^{-x G(T)} ] \, =\, \widehat{\bm \pi} ( x, 0) \, =\,  \widehat{\bm \pi} ( 0 , x) \, =\,  \frac{\,\lambda \,}{\,x + \lambda \,} [  2 - \widehat{\bm \pi} (x, x ) ] \, ; \quad x \ge 0 \, ; 
\end{equation*}
cf.\ (2.73) in \cite{ichiba_karatzas_degenerate_22-arxiv}. It follows from the exponential moment \eqref{eq: App4} that $\, G(T) \,$ has positive stationary density 
\begin{equation}
\label{eq:density_G_sym}
\mathbb P^{\bm \pi} ( G(T) \in {\mathrm d} u) \, =\,  \lambda e^{-\lambda u} \Big( 2 - \int^{u}_{0} e^{\lambda z} {\bm \sigma}(z) {\mathrm d} z \Big) {\mathrm d} u \, =\,   \lambda e^{-\lambda u} \, \Big(  \int^{\infty}_{u} e^{\lambda z} {\bm \sigma}(z) {\mathrm d} z \Big) \, {\mathrm d} u \, 
\end{equation}
for $u \ge 0 $. In particular, the stationary ``survival function'' of $G(T)$ is 
\begin{align*}
\mathbb P^{\bm \pi} ( G(T) \ge u) &= \int^\infty_u \Big ( \lambda e^{-\lambda v} \int^\infty_v e^{\lambda z} {\bm \sigma} (z) \mathrm d z \Big) \mathrm d v = \int^\infty_u ( e^{-\lambda (u - z)} - 1) {\bm \sigma} (z) \mathrm d z  \\
&= \mathbb E^{\bm \pi} \Big [ ( e^{\lambda (G(T) + H(T) - u)} - 1) \cdot {\bf 1}_{\{G(T) + H(T) > u \}} \Big]\, ; \quad u \ge 0 \, . 
\end{align*}

As will be proved in the following result, the distributions of $G(T)+H(T)$ and $2G(T)+H(T)$ are both given by those of infinite sums $\sum {\bm \varepsilon}_{k}$ of exponential variables, the only difference being that in the first case the sum runs over $ k\in \mathbb N$, while in the second case the index $0$ should be added.
\begin{proposition}
\label{prop:distribution_2G+H}
    Assume \eqref{eq:symcond}. Under the stationary distribution ${\bm \pi}$, the probability density function of the sum $\,2G(T)+H(T)\,$ (and of $\,G(T)+2H(T)\,$ by symmetry reasons)\ is that of the infinite sum $\, \sum_{k=0}^{\infty} {\bm \varepsilon}_{k}\,$ of independent exponential random variables $\, \{ {\bm \varepsilon}_{k}\}_{k \in \mathbb N} \,$ with respective  parameters $\, {\bm \ell}_{k}\,$ given by \eqref{eq: App2}.
\end{proposition}
\begin{proof}
Evaluating the functional equation \eqref{eq:funceq} at $(2y,y)$, we immediately obtain
\begin{equation}
\label{eq:Lap_trans_(2y,y)}
     \mathbb{E}^\boldpi \left( e^{-y(2G(T)+H(T))} \right) = \hatpi (2y,y) = \frac{3}{2}\frac{\hatnu_1(y)}{y+3\lambda}.
\end{equation}
Using \eqref{eq: App1} and noticing with \eqref{eq: App2} that  ${\bm \ell}_{0}=3\lambda$, we obtain the announced result.
\end{proof}

\subsection{The non-symmetric case}

Results involving infinite sums of random variables are also shown in the non-symmetric case. We prove the second part of Theorem~\ref{thm:sum_exp_sum}.

Observe that Theorem~\ref{thm:sum_exp_sum} \ref{thm4.2} is an extension of \eqref{eq: App1} to the non-symmetric case. Indeed, specializing $\lambda_1=\lambda_2$ and thus $\mu_i=\frac{1}{2}$, we immediately obtain via \eqref{eq:equality_orbits} and \eqref{eq: App2} and  that the parameters $k(k+\mu_i)(\lambda_1+\lambda_2)$, for $k\in\mathbb Z\setminus \{-1,0,1\}$, reduce to the ${\bm \ell}_{k}$, for $k\geq 1$. 

\begin{proof}[Proof of Theorem~\ref{thm:sum_exp_sum}\ref{thm4.2}]
We prove that
\begin{equation}
\label{eq:dist_nu_i_inf_conv}
    \,\widehat{\bm \nu}_{1}(y)\, \, =\,  \frac{2}{3}(2\lambda_1+\lambda_2)\prod_{k\in\mathbb Z\setminus \{-1,0,1\}} \frac{k(k+\mu_1)(\lambda_1+\lambda_2)}{y + k(k+\mu_1)(\lambda_1+\lambda_2)}  \, ; \quad y \in [0, \infty). 
\end{equation}
(Note that the prefactor $\frac{2}{3}(2\lambda_1+\lambda_2)$ simply corresponds to the value of $\hatnu_1(0)$ already obtained in \eqref{eq:value_hatnu_1(0)}.)

Let us start from the identity \eqref{eq:expression_hatnu_1_nonsym} proved in Theorem~\ref{thm:explicit_general}.
We have the Weierstrass factorization
\begin{equation*}
    \cos z-\cos z_0 = (1-\cos z_0) \prod_{k\in\mathbb Z}\left(1-\Bigl(\frac{z}{z_0+2k\pi}\Bigr)^2\right),
\end{equation*}
which extends \eqref{eq:cos_infinite_product_formula}.
In particular,
\begin{equation*}
    \cos \bigl(\pi\sqrt{\mu_1^2-4y}\bigr) - \cos(\pi\mu_1) = (1-\cos(\pi\mu_1)) \prod_{k\in\mathbb Z} \frac{y+k(k+\mu_1)}{(k+\frac{\mu_1}{2})^2}.
\end{equation*}
This implies that the quantity appearing in \eqref{eq:expression_hatnu_1_nonsym} (after replacing $y$ by $(\lambda_1+\lambda_2)y$)\ is equal to
\begin{equation*}
    \frac{y(y+1+\mu_1)(y+1-\mu_1)}{\cos \bigl(\pi\sqrt{\mu_1^2-4y}\bigr) - \cos(\pi\mu_1)} = \frac{\bigl((\frac{\mu_1}{2}-1)\frac{\mu_1}{2}(\frac{\mu_1}{2}+1)\bigr)^2}{1-\cos(\pi\mu_1)}\prod_{k\in\mathbb Z\setminus \{-1,0,1\}} \frac{(k+\frac{\mu_1}{2})^2}{y+k(k+\mu_1)}.
\end{equation*}
Using the fact that 
\begin{equation*}
    \prod_{k\in\mathbb Z\setminus \{-1,0,1\}} \frac{(k+\frac{\mu_1}{2})^2}{k(k+\mu_1)} = \frac{1-\cos(\pi\mu_1)}{\pi\sin(\pi\mu_1)} \frac{2(1-\mu_1^2)}{(\frac{\mu_1}{2}-1)^2\mu_1(\frac{\mu_1}{2}+1)^2},
\end{equation*}
we obtain that 
\begin{equation*}
    \frac{y(y+1+\mu_1)(y+1-\mu_1)}{\cos \bigl(\pi\sqrt{\mu_1^2-4y}\bigr) - \cos(\pi\mu_1)} = \frac{\frac{\mu_1}{2}(1-\mu_1^2)}{\pi\sin(\pi\mu_1)}\prod_{k\in\mathbb Z\setminus \{-1,0,1\}} \frac{k(k+\mu_1)}{y+k(k+\mu_1)}.
\end{equation*}
Plugging the above identity in \eqref{eq:expression_hatnu_1_nonsym}, we obtain
\begin{equation*}
    \hatnu_1({y})=\frac{4\pi}{3\lambda_1\lambda_2}\sin(\pi\mu_1)(\lambda_1+\lambda_2)^3\frac{\frac{\mu_1}{2}(1-\mu_1^2)}{\pi\sin(\pi\mu_1)}\prod_{k\in\mathbb Z\setminus \{-1,0,1\}} \frac{k(k+\mu_1)(\lambda_1+\lambda_2)}{y+k(k+\mu_1)(\lambda_1+\lambda_2)},
\end{equation*}
which after simplification coincides exactly with \eqref{eq:dist_nu_i_inf_conv}.
\end{proof}

As a direct consequence of Theorem~\ref{thm:sum_exp_sum} \ref{thm4.2}, one has a generalization of Proposition~\ref{prop:distribution_2G+H} to the non-symmetric case:
\begin{equation*}
     \mathbb{E}^\boldpi \left( e^{-y(2G+H)} \right) = \hatpi (2y,y) = \frac{3}{2}\frac{\hatnu_1(y)}{y+2\lambda_1+\lambda_2}=\prod_{k\in\mathbb Z\setminus \{-1,0\}} \frac{k(k+\mu_1)(\lambda_1+\lambda_2)}{y + k(k+\mu_1)(\lambda_1+\lambda_2)}.
\end{equation*}

It is also possible to deduce a result like \eqref{eq:density_G_sym} in the non-symmetric case. Indeed, it follows from the Laplace transforms 
\begin{equation*}
\begin{split}
    \mathbb E ^{\bm \pi} [ e^{-x G(T)} ] &=\hatpi(x,0) = \frac{\frac{2}{3}(2\lambda_1+\lambda_2)-\frac{\hatnu_2(x)}{2}}{x+\lambda_1}\,; \quad x \ge 0 \, ,  \\
    \mathbb E ^{\bm \pi} [ e^{-y H(T)} ] &=\hatpi(0,y) = \frac{\frac{2}{3}(\lambda_1+2\lambda_2)-\frac{\hatnu_1(y)}{2}}{y+\lambda_2} \, ; \quad y \ge 0 \, ,
\end{split}
\end{equation*}
that we can obtain the marginal density functions by the inverse Laplace transforms
\begin{equation} \label{eq: marginalpdfGH}
\begin{split}
\mathbb P ^{\bm \pi} ( G(T) \in \mathrm d u) & = 
 \frac{1}{3} (\lambda_1 + \lambda_2) e^{-\lambda_1 u} \Big(  2 (1 + \mu_1) - 
 (1+\mu_2) \int^u_0 e^{\lambda_1 z} {\widetilde{\nu}_2} (z) \mathrm d z \Big) \mathrm d u  \, , \\
 \mathbb P ^{\bm \pi} ( H(T) \in \mathrm d u) & = 
\frac{1}{3} (\lambda_1 + \lambda_2) e^{-\lambda_2 u} \Big(  2 (1 + \mu_2)  - 
 (1+\mu_1) \int^u_0 e^{\lambda_2 z} {\widetilde{\nu}_1} (z) \mathrm d z \Big) \mathrm d u \, 
\end{split}
\end{equation}
for $\, u \ge 0\, $, with $\, \mu_i = \lambda_i / (\lambda_1 + \lambda_2)\, $, $i=1, 2$ as in \eqref{eq:def_mu}. Here, $\widetilde{\nu}_i$ is the probability density function  of the probability measure $\, ( 2 (\lambda_1 + \lambda_2)(1+\mu_i)/3) \, {\bm \nu}_i (\cdot)\, $ on the positive real line for $i =1, 2$ in Theorem~\ref{thm:sum_exp_sum} \ref{thm4.2}. Note that the right-hand formulas in \eqref{eq: marginalpdfGH} are positive for $u \ge 0$, because for the first formula in \eqref{eq: marginalpdfGH}, as in \eqref{eq: App4}, we have 
\begin{equation*}
\begin{split}
\int^\infty_0 e^{\lambda_1 z} \widetilde{\nu}_2 (z) \mathrm d z &= \prod_{k \in \mathbb Z \setminus\{-1, 0, 1\}} \frac{k(k+\mu_2)(\lambda_1+\lambda_2)}{k(k+\mu_2)(\lambda_1 + \lambda_2) - \lambda_1}\\
&= \prod_{k \in \mathbb Z \setminus\{-1, 0, 1\}} \frac{k(k+\mu_2)}{k(k+\mu_2) - \mu_1}  = \frac{2(1+\mu_1)}{1+\mu_2} \, . 
\end{split}
\end{equation*}
We used here the relation $\mu_1+\mu_2=1$, and the telescopic structure
\begin{equation*}
    \frac{k(k+\mu_2)}{k(k+\mu_2) - \mu_1} = \frac{k}{k+1} \frac{k+1-\mu_1}{k-\mu_1} \, ; \quad k \in \mathbb Z \setminus \{-1, 0, 1\} 
\end{equation*}
in the last equality.

\section{Bivariate density via the compensation approach}
\label{sec:compensation}

\subsection{A PDE satisfied by the stationary distribution}
Let us denote by $\pi(u,v)$ the density of the invariant measure, and recall the parameters $\lambda_1,\lambda_2$ introduced in~\eqref{eq:def_lambda_1_2}.
In the manner of \cite[(8.5)]{harrison_reiman_distribution_81}, we may state the following partial differential equation (PDE) satisfied by this probability density function:
\begin{equation}
\label{eq:edp}
\begin{cases}
\mathcal{G}^* \pi (u,v) =0 & \text{for } (u,v)\in\mathbb{R}_+^2,
\\
\partial_{R^*_1}\pi (0,v) +2\lambda_1 \pi (0,v) =0
& \text{for } v\in\mathbb{R}_+,
\\
\partial_{R^*_2}\pi (u,0) +2\lambda_2 \pi (u,0) =0
& \text{for } u\in\mathbb{R}_+,
\end{cases}
\end{equation}
where we denote
\begin{equation}
\label{eq:dual_operators}
\left\{\begin{array}{lcl}
   \mathcal{G}^*&=&\displaystyle \left(\frac{\partial}{\partial x}-\frac{\partial}{\partial y}\right)^2+\lambda_1\frac{\partial}{\partial x}+\lambda_2\frac{\partial}{\partial y},\medskip\\
 R&=&\displaystyle\begin{pmatrix}
1 & -\frac{1}{2} \\ 
-\frac{1}{2} & 1
\end{pmatrix},\medskip\\
R^*&=&\displaystyle2\Sigma-R \text{diag}(R)^{-1}\text{diag}(\Sigma)
=\begin{pmatrix}
2 & -3 \\ 
-3 & 2
\end{pmatrix},
\\\Sigma&=&\medskip \displaystyle\begin{pmatrix}
2 & -2 \\ 
-2 & 2
\end{pmatrix}.
\end{array}\right.
 \end{equation}

\subsection{The Compensation Approach: Basic Principle, Questions} 
 
Using our notations \eqref{eq:edp}--\eqref{eq:dual_operators}, let us introduce the sets of functions
\begin{equation*}
\left\{\begin{array}{rcl}
H_0&=&\displaystyle \{f \in \mathcal{C}^2(\mathbb{R}_+^2): \mathcal{G}^* f  =0 \},\smallskip\\
H_1&=&\displaystyle H_0 \cap \{f \in \mathcal{C}^2(\mathbb{R}_+^2):
\partial_{R^*_1}f (0,\cdot ) +2\lambda_1 f (0,\cdot) =0 \},\smallskip\\
H_2&=&\displaystyle H_0 \cap \{f \in \mathcal{C}^2(\mathbb{R}_+^2): 
\partial_{R^*_2}f (\cdot,0) +2\lambda_2 f (\cdot,0) =0 \},
\end{array}\right.
\end{equation*}
and observe that the requirement $\pi\in H_1\cap H_2$ is equivalent to the system \eqref{eq:edp}. 

We look now for exponential functions in $H_0$, $H_1$ and $H_2$. We have 
\begin{align}
e^{-au-bv}\in H_0
&\Longleftrightarrow
K^*(a,b):=(a-b)^2-\lambda_1a-\lambda_2b=0
\label{eq:K*}
\\ &\Longleftrightarrow
(a,b)\in \mathcal{P}^*:= \{(x,y)\in\mathbb{R}^2: (x-y)^2-\lambda_1x-\lambda_2y=0 \}.
\label{eq:P*}
\end{align}
Similarly, 
\begin{align}
\nonumber
e^{-au-bv}\in H_1 &\Longleftrightarrow
(a-b)^2-\lambda_1a-\lambda_2b=0
\quad\text{and}\quad
2a-3b-2\lambda_1=0
\\ \nonumber&\Longleftrightarrow
(a,b)\in\mathcal{P}^*\cap \mathcal{L}_1, \text{ where } 
\mathcal{L}_1:=\{(x,y)\in\mathbb{R}^2: 2x-3y-2\lambda_1=0 \}
\\ &\label{eq:a0b0}\Longleftrightarrow
(a,b)=(\lambda_1,0)
\quad\text{or}\quad
(a,b)=(a_0,b_0):=(4\lambda_1+
6\lambda_2,2\lambda_1+4\lambda_2),
\end{align}
and symmetrically
\begin{align}
\nonumber
e^{-a'u-b'v}\in H_2 &\Longleftrightarrow
(a'-b')^2-\lambda_1a'-\lambda_2b'=0
\quad\text{and}\quad
2b'-3a'-2\lambda_2=0
\\ &\Longleftrightarrow
(a',b')\in\mathcal{P}^*\cap \mathcal{L}_2, \text{ with } 
\mathcal{L}_2:=\{(x,y)\in\mathbb{R}^2: 2y-3x-2\lambda_2=0 \}
\label{eq:L2}
\\\nonumber &\Longleftrightarrow
(a',b')=(0,\lambda_2)
\quad\text{or}\quad
(a',b')=(a'_0,b'_0):=(4\lambda_1+
2\lambda_2,6\lambda_1+
4\lambda_2).
\end{align}
The parabola $\mathcal{P}^*$ in \eqref{eq:P*}, the lines $\mathcal{L}_1$ and $\mathcal{L}_2$ in~\eqref{eq:a0b0}--\eqref{eq:L2}, the points $(a_0,b_0)$ and $(a'_0,b'_0)$ defined above, can be visualized on Figure~\ref{fig:compensation}.

 The main idea of the compensation approach \cite{adan_wessels_zijm_compensation_93} is to start with an exponential function in $H_1$ (resp.\ $H_2$) and to add another exponential, so that the sum of the two terms belongs to $H_2$ (resp.\ $H_1$). This step is the first compensation, but this sum of two functions is still not in $H_1$ (resp.\ $H_2$). Therefore, we have to compensate again with another exponential term, and so on. We eventually compensate with an infinite sum of exponential functions in such a way that the final sum be in $H_1\cap H_2$. The following equation is a visualization of this approach:
\begin{equation}
\label{eq:value_p(u,v)}
   p(u,v):=\overunderbraces{&\br{1}{\in H_1}& &\br{3}{\in H_1}& &\br{3}{\in H_1}}%
  {&c_0e^{-a_0u-b_0v} &+& c_1e^{-a_1u-b_1v}  &+&c_2e^{-a_2u-b_2v} &+&c_3e^{-a_3u-b_3v} &+ & c_4e^{-a_4u-b_4v} &+ \cdots}%
  {& \br{3}{\in H_2} & &\br{3}{\in H_2}}.
\end{equation}
A symmetric construction holds for a first term in $H_2$:
\begin{equation}
\label{eq:value_pp(u,v)}
   p'(u,v):=\overunderbraces{&\br{3}{\in H_1}& &\br{3}{\in H_1}}%
  {&c'_0e^{-a'_0u-b'_0v} &+& c'_1e^{-a'_1u-b'_1v}  &+&c'_2e^{-a'_2u-b'_2v} &+&c'_3e^{-a'_3u-b'_3v} &+ & c'_4e^{-a'_4u-b'_4v} &+ \cdots}%
  {& \br{1}{\in H_2} & &\br{3}{\in H_2} & &\br{3}{\in H_2}}.
\end{equation}
This approach raises several questions, which we will answer in the remainder of the paper:
\begin{itemize}
    \item What are the explicit values of the constants $a_n$, $b_n$ and $c_n$ in \eqref{eq:value_p(u,v)}? (And a symmetric question for \eqref{eq:value_pp(u,v)}.) We will answer this question in Section~\ref{subsec:computation_constants}.
    \item For any values of $C$ and $C'$, the linear combination
    \begin{equation}
\label{eq:convex_combin_ppp}
   Cp(u,v)+C'p'(u,v)
\end{equation}
is a solution to the PDE \eqref{eq:edp}. Is it possible to find the invariant distribution among these infinitely-many solutions? If yes, how to adjust the constants $C$ and $C'$ so as to find the unique invariant distribution? We will answer this question in Section~\ref{subsec:computation_combination}.
\end{itemize}

\begin{remark}[Starting points of the sequences $(a_n)_{n\geq0}$ and $(b_n)_{n\geq0}$ of the compensation procedure]

Although the point $(\lambda_1,0)\in \mathcal{P}^*\cap \mathcal{L}_1$ is formally a solution to \eqref{eq:a0b0}, it is not possible to choose it as a starting point of the compensation approach. Indeed, the exponential $e^{\lambda_1u}$ does not converge to $0$ when $v\to\infty$ and $u=0$. The procedure must thus be initialized at $(a_0,b_0)\in \mathcal{P}^*\cap \mathcal{L}_1$, see again \eqref{eq:a0b0}. A similar remark applies to the point $(0,\lambda_2)$. 
\end{remark}

\begin{remark}[Skew symmetry]
The intersection $\mathcal{P}^*\cap\mathcal{L}_1\cap\mathcal{L}_2$ of the sets introduced in \eqref{eq:P*}, \eqref{eq:a0b0} and \eqref{eq:L2} is empty most of the time, except in the so-called skew symmetric case. For example, when
\begin{equation*}
R=\begin{pmatrix}
1 & -\frac{1}{2} \\ 
-\frac{3}{2} & 1
\end{pmatrix},
\end{equation*}
we have $\mathcal{P}^*\cap\mathcal{L}_1\cap\mathcal{L}_2=\{(3\lambda_2+2\lambda_1,2\lambda_2+\lambda_1)\}$ and we find again \cite[Eq.~(A.22)]{ichiba_karatzas_degenerate_22}
\begin{equation*}
\pi(u,v)=(3\lambda_2+2\lambda_1)(2\lambda_2+\lambda_1)e^{-(3\lambda_2+2\lambda_1)u-(2\lambda_2+\lambda_1)v}.
\end{equation*}
\end{remark}

\subsection{Computation of the compensation constants}
\label{subsec:computation_constants}

We now obtain explicit formulas for the sequences $(a_n)_{n\geq 0}$, $(b_n)_{n\geq 0}$ and $(c_n)_{n\geq 0}$ appearing in~\eqref{eq:value_p(u,v)} (and Theorem~\ref{thm:maindensity}). We will prove the following:
\begin{proposition}
\label{prop:values_an_bn}
With $(a_0,b_0)$ determined in~\eqref{eq:a0b0} the sequence defined by, for all $n\geq0$,
\begin{equation}
\label{eq:values_(a_n,b_n)}
\begin{cases}
    a_{2n}=a_0+2n(a_0-b_0)+n^2\lambda_1+n(n+1)\lambda_2 ,
    \\
    b_{2n}=b_0+2n(a_0-b_0)+n(n-1)\lambda_1+n^2\lambda_2 ,
\end{cases}
\quad\text{and}\quad
(a_{2n+1},b_{2n+1})= (a_{2n},b_{2n+2})
\end{equation}
satisfies the following statement:
\begin{equation*}
   e^{-a_n u -b_n v}\in H_0, \quad \forall \ n\in \mathbb{N}.
   \end{equation*}
\end{proposition}
\begin{proof}
As explained in~\eqref{eq:K*}, $e^{-a_n u -b_n v}\in H_0$ if and only if $K^*(a,b)=0$.  With a few simple but tedious computation, we can verify that for all $n\in \mathbb{N}$ we have $K^*(a_n,b_n)=0$ which concludes the proof.

    We also give a more constructive procedure that enabled us to determine these sequences. We need to introduce $\zeta$ and $\eta$ two \emph{automorphisms} of the parabola $\mathcal{P}^*$ in~\eqref{eq:P*}, defined by
\begin{equation*}
\zeta (x,y)= (x,2x-y+\lambda_2)
\quad \text{and} \quad
\eta (x,y) = (2y-x+\lambda_1,y).
\end{equation*}
By construction, these satisfy, for $(x,y)\in \mathcal{P}^*$ such that $K^*(x,y)=0$, that $\zeta(x,y)\in \mathcal{P}^*$ and $\eta(x,y)\in \mathcal{P}^*$, i.e.
\begin{equation*}
K^*\bigl(\zeta(x,y)\bigr)=K^*\bigl(\eta(x,y)\bigr)=0.
\end{equation*}
These automorphisms can be visualized on Figure~\ref{fig:compensation}. They allow us to define recursively the sequences $(a_n)_{n\geq 0}$ and $(b_n)_{n\geq 0}$:
\begin{equation*}
\left\{\begin{array}{rcl}
(a_{2n},b_{2n})&=& (\eta\zeta)^n(a_{0},b_0),\\
(a_{2n+1},b_{2n+1})&=& \zeta(\eta\zeta)^n(a_{0},b_0)= (a_{2n},b_{2n+2}).
\end{array}\right.
\end{equation*}
We have $\eta\zeta(x,y)=(3x-2y+2\lambda_2+\lambda_1,2x-y+\lambda_2)$; and a straightforward computation allows us to verify the recurrence relation
\begin{equation*}
   \eta\zeta (a_{2n},b_{2n})=(3a_{2n}-2b_{2n}+2\lambda_2+\lambda_1,2a_{2n}-b_{2n}+\lambda_2)=(a_{2n+2},b_{2n+2}),
\end{equation*}
which proves~\eqref{eq:values_(a_n,b_n)}.
\end{proof}

\begin{proposition}
\label{prop:values_cn}
The sequence $(c_n)_{n\geq0}$ defined by induction as follows: $c_0=1$ and for $n\geq0$,
\begin{equation}
\label{eq:values_(c_n)}
\left\{\begin{array}{rcl}
c_{2n+1}&=&\displaystyle -c_{2n}\frac{3a_{2n}-2b_{2n}+2\lambda_2}{3a_{2n+1}-2b_{2n+1}+2\lambda_2},\medskip\\
c_{2n+2}&=&\displaystyle -c_{2n+1}\frac{-2a_{2n+1}+3b_{2n+1}+2\lambda_1}{-2a_{2n+2}+3b_{2n+2}+2\lambda_1},
\end{array}\right.
\end{equation}
satisfies the compensation approach of~\eqref{eq:value_p(u,v)}, i.e.\ for all $n \in \mathbb{N}$ we have
\begin{equation*}
\begin{cases}
    c_{2n}e^{-a_{2n} u -b_{2n} v} + c_{2n+1}e^{-a_{2n+1} u -b_{2n+1} v} \in H_2,
\\
c_{2n+1}e^{-a_{2n+1} u -b_{2n+1} v}+c_{2n+2}e^{-a_{2n+2} u -b_{2n+2} v} \in H_1.
\end{cases}
\end{equation*}
\end{proposition}

\begin{proof}
For all $n \in \mathbb{N}$ let us define 
\begin{equation*}
    f_{2n}(u,v):=c_{2n}e^{-a_{2n} u -b_{2n} v} + c_{2n+1}e^{-a_{2n+1} u -b_{2n+1} v}.
\end{equation*}
Let us recall that 
\begin{equation*}
   H_2= H_0 \cap \{f \in \mathcal{C}^2(\mathbb{R}_+^2): 
\partial_{R^*_2}f (\cdot,0) +2\lambda_2 f (\cdot,0) =0 \}.
\end{equation*}
Proposition~\ref{prop:values_an_bn} implies that $f_{2n} \in H_0$.
We also have
\begin{multline*}
   \partial_{R^*_2}f_{2n}(u,0) +2\lambda_2 f_{2n} (u,0) =\\c_{2n}(3a_{2n}-2b_{2n}+2\lambda_2)e^{a_{2n}u}+c_{2n+1}(3a_{2n+1}-2b_{2n+1}+2\lambda_2)e^{a_{2n+1}u}
   =0,
\end{multline*}
remembering that $a_{2n}=a_{2n+1}$ and the relation~\eqref{eq:values_(c_n)}. Then $f_{2n}\in H_2$ and a similar reasoning shows that $f_{2n+1}(u,v):=c_{2n+1}e^{-a_{2n+1} u -b_{2n+1} v}+c_{2n+2}e^{-a_{2n+2} u -b_{2n+2} v} \in H_1$ and thus completes the proof.
\end{proof}

\begin{corollary}
\label{cor:values_polynomial_c_n}
The sequence $(c_n)_{n\geq0}$ is piecewise polynomial: if $n$ is even, then
\begin{equation}
\label{eq:values_c_n_even}
       c_n = (n + 2) (n + 4)^2(n + 6)
    \frac{(n + 2\mu_2 )(n + 2  + 2\mu_2 )(n + 4  + 2\mu_2 )(n + 6  + 2\mu_2 )}
{3072\mu_2(1+\mu_2   )(2+ \mu_2 )(3+ \mu_2 ) },\normalsize
\end{equation}
and if $n$ is odd
\begin{equation}
\label{eq:values_c_n_odd}
c_n = -(n + 1)(n + 3)(n + 5)(n + 7)
\frac{(n + 1 + 2\mu_2 )(n + 3  + 2\mu_2 )^2(n + 5  + 2\mu_2 )}{3072\mu_2(1+\mu_2   )(2+\mu_2 )(3+\mu_2 )}.
\end{equation}
As $n\to\infty$, 
\begin{equation}    \label{eq:asymptotics_c_n}
c_n \sim \frac{1}{3072\mu_2(1+\mu_2   )(2+\mu_2 )(3+\mu_2 ) }(-1)^nn^8.
\end{equation}
If $\lambda_1=\lambda_2$, then both \eqref{eq:values_c_n_even} and \eqref{eq:values_c_n_odd} reduce to
\begin{equation*}
    c_n =(-1)^n\frac{(n + 1)(n + 2)(n + 3)(n + 4)^2(n + 5)(n + 6)(n + 7)}{20160}= (-1)^n\binom{n+7}{7}\frac{n+4}{4}.
\end{equation*}
\end{corollary}

\begin{figure}[hbtp]
\centering
\includegraphics[scale=0.7]{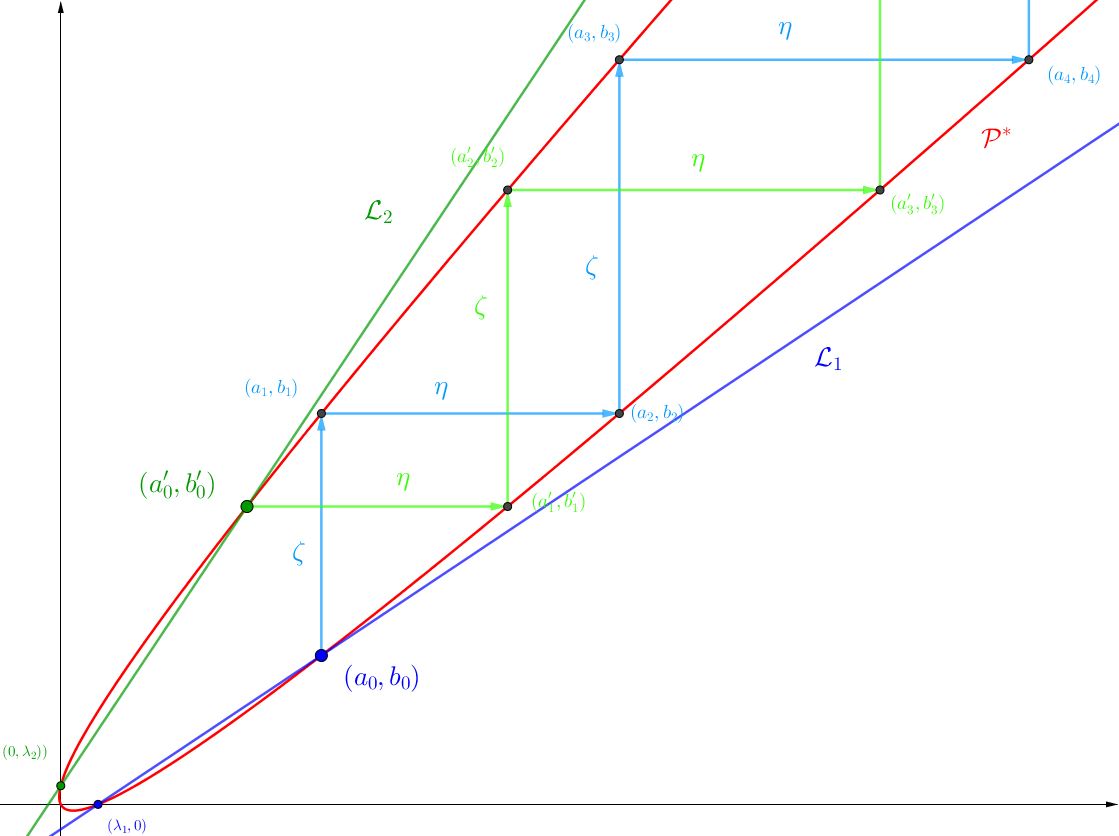}
\caption{The parabola $\mathcal{P}^*$ is represented in red, the line $\mathcal{L}_1$ in blue, $\mathcal{L}_2$ in green, the starting point $(a_0,b_0)$ in blue and $(a'_0,b'_0)$ in green. The automorphisms $\zeta$ and $\eta$ on the parabola allow us to define the sequences $(a_n,b_n)_{n\geq0}$ and $(a'_n,b'_n)_{n\geq0}$.}
\label{fig:compensation}
\end{figure}

\begin{proof}
We start by reformulating the recurrence relation \eqref{eq:values_(c_n)}. We have
\begin{equation*}
    \frac{3a_{2n}-2b_{2n}+2\lambda_2}{3a_{2n+1}-2b_{2n+1}+2\lambda_2} = \frac{(n + 4)(n + 2 + 3\mu_2)}{(n + 2)(n + \mu_2)}.
\end{equation*}
Similarly, we have
\begin{equation*}
   \frac{-2a_{2n+1}+3b_{2n+1}+2\lambda_1}{-2a_{2n+2}+3b_{2n+2}+2\lambda_1} =\frac{(n + 3)(n + 4 + \mu_2)}{(n + 1)(n + 2+ \mu)}.
\end{equation*}
We immediately obtain via \eqref{eq:values_(c_n)} that
\begin{equation}
\label{eq:rec_2_c_n}
    c_{2n+2}=c_{2n}\frac{(n + 3)(n + 4)(n  + 4  + \mu_2 )}{(n + 1)(n + 2)(n  + \mu_2 )},
\end{equation}
which shows that $c_{2n}$ admits a telescopic structure. More precisely, denoting $T_n =(n + 1)(n + 2)$ and $U_n =n  + \mu_2 $, \eqref{eq:rec_2_c_n} can be rewritten as
\begin{equation*}
   c_{2n+2}=c_{2n} \frac{T_{n+2}}{T_n}\frac{U_{n+4}}{U_n}.
\end{equation*}
We conclude that $c_{2n}=\frac{T_nT_{n+1}}{T_0T_1}\frac{U_nU_{n+1}U_{n+2}U_{n+3}}{U_0U_1U_2U_3}$. Replacing $n$ by $\frac{n}{2}$, this coincides with the value of $c_n$ announced in \eqref{eq:values_c_n_even}.
The proof of \eqref{eq:values_c_n_odd} would be similar.
\end{proof}

Symmetric formulas hold for the sequences $(a'_n)_{n\geq 0}$, $(b'_n)_{n\geq 0}$ and $(c'_n)_{n\geq 0}$ in~\eqref{eq:value_pp(u,v)}:
\begin{equation}
\label{eq:values_(ap_n,bp_n)}
   (a'_{2n},b'_{2n}) = \bigl(a'_0+2n(b'_0-a_0')+n^2\lambda_1+n(n-1)\lambda_2, b'_0+2n(b'_0-a_0')+n(n+1)\lambda_1+n^2\lambda_2\bigr),
\end{equation}
and $(a'_{2n+1},b'_{2n+1})= (a'_{2n+2},b'_{2n})$.
We finally introduce $c_0'=1$ and we have
\begin{equation}
\label{eq:values_(cp_n)}
\left\{\begin{array}{rcl}
c'_{2n+1}&=&\displaystyle -c'_{2n}\frac{3b'_{2n}-2a'_{2n}+2\lambda_1}{3b'_{2n}-2a'_{2n+1}+2\lambda_1},\medskip\\
c'_{2n+2}&=&\displaystyle -c'_{2n+1}\frac{-2b'_{2n+1}+3a'_{2n+1}+2\lambda_2}{-2b'_{2n+2}+3a'_{2n+2}+2\lambda_2}.
\end{array}\right.
\end{equation}
The sequence $(c'_n)_{n\geq 0}$ admits the same exact and asymptotic expressions as $(c_n)_{n\geq 0}$, provided $\lambda_1$ and $\lambda_2$ are interchanged.

\subsection{Computation of the convex combination}
\label{subsec:computation_combination}

\begin{proposition}
\label{prop:comp_p(u,0)}
The function $p(u,v)$ in \eqref{eq:value_p(u,v)} evaluated at $v=0$ is equal to
\begin{multline*}
    p(u,0) =-\frac{1}{3\mu_2(1+\mu_2   )(2+\mu_2 )(3+\mu_2 ) }\times\\ \sum_{n\geq 2}(n-1)n(n+1)(n - 1 + \mu)
    (n + \mu)(n + 1 + \mu)(n + \mu/2)\exp(-(n^2+\mu n)(\lambda_1+\lambda_2)u).
\end{multline*}
\end{proposition}
\begin{proof}
We start from the expression of $p(u,v)$ given in \eqref{eq:value_p(u,v)}. Since $a_{2n}=a_{2n+1}$ (see Proposition~\ref{prop:values_an_bn}), we can group the terms as follows:
\begin{equation*}
    p(u,0) = \sum_{n\geq 0} (c_{2n}-c_{2n+1})e^{-a_{2n}u}.
\end{equation*}
Using the expression of $c_{2n}$ and $c_{2n+1}$ given in Proposition~\ref{cor:values_polynomial_c_n} and after some simplification, we obtain
\begin{multline*}
    c_{2n}-c_{2n+1} = -\frac{1}{3\mu_2(1+\mu_2   )(2+\mu_2 )(3+\mu_2 ) }\times\\ (n+1)(n+2)(n+3)(n + 1 + \mu)
    (n +2+ \mu)(n + 2 + \mu)(n + 2+\mu/2).
\end{multline*}
Moreover, one can reformulate $a_{2n} =\bigl((n+2)^2+\mu (n+2)\bigr)(\lambda_1+\lambda_2) $. We then immediately deduce Proposition~\ref{prop:comp_p(u,0)}.
\end{proof}
Similarly, we have:
\begin{proposition}
\label{prop:comp_p'(u,0)}
The function $p'(u,v)$ in \eqref{eq:value_pp(u,v)} evaluated at $v=0$ is equal to
\begin{multline*}
    p'(u,0) = -\frac{1}{3\mu_1(1+\mu_1)(2+\mu_1)(3+\mu_1) }\times\\ \sum_{n\leq -2}(n-1)n(n+1)(n - 1 + \mu)
    (n + \mu)(n + 1 + \mu)(n + \mu/2)\exp(-(n^2+\mu n)(\lambda_1+\lambda_2)u).
\end{multline*}
\end{proposition}

As a consequence of Propositions~\ref{prop:comp_p(u,0)} and \ref{prop:comp_p'(u,0)}, we obtain that 
there exists a unique choice of constants $C$ and $C'$, namely formula~\eqref{eq:choice_C_C'} of  Corollary~\ref{cor:nuintermofp}, such that the convex combination
$Cp(u,0)+C'p'(u,0)$ in \eqref{eq:convex_combin_ppp} is equal to the formula \eqref{eq:formula_boundary_density_original_question_non_symmetric} for the density function $\nu_i(u)$ given in Theorem~\ref{cor:boundary_density_non-sym}. 

We furthermore conjecture that it must be the unique choice of $C$ and $C'$ such that $Cp(u,v)+C'p'(u,v)$ is a positive function. This should follow from a result of uniqueness of positive solutions of the PDE~\eqref{eq:edp}.

\begin{corollary}[Values of the constants $C$ and $C'$]
\label{cor:nuintermofp}
Taking
\begin{equation}
\label{eq:choice_C_C'}
    C:=\frac{4(\lambda_1+2\lambda_2   )(2\lambda_1  + 3\lambda_2 )(3\lambda_1  + 4\lambda_2 ) }{\lambda_1}\medskip \, \text{ and } \, 
    C':=\frac{4(\lambda_2+2\lambda_1   )(2\lambda_2  + 3\lambda_1 )(3\lambda_2  + 4\lambda_1 ) }{\lambda_2},
\end{equation}
we have 
\begin{equation*}
\nu_2(u)=Cp(u,0)+C'p'(u,0)
\quad\text{and}\quad
\nu_1(v)=Cp(0,v)+C'p'(0,v).
\end{equation*}
\end{corollary}
\begin{proof}
    This follows directly from the Propositions~\ref{prop:comp_p(u,0)} and \ref{prop:comp_p'(u,0)} and Theorem~\ref{cor:boundary_density_non-sym}.
\end{proof}
It is interesting to note that by construction, remembering~\eqref{eq:asymptotics_c_n}, $C$ and $C'$ are such that 
\begin{equation*}
   c_n/c'_n \underset{n\to\infty}{\longrightarrow} C'/C .\end{equation*}

\begin{statement}[Statement equivalent to Theorem~\ref{thm:maindensity}]
\label{statement}
The bivariate density $\pi(u,v)$ is given by $Cp(u,v)+C'p'(u,v)$, with $C$ and $C'$ as in \eqref{eq:choice_C_C'}.
\end{statement}

\begin{proof}[Proof of Statement~\ref{statement} and Theorem~\ref{thm:maindensity}]
    Let $f(u,v)=Cp(u,v)+C'p'(u,v)$, which satisfies the PDE~\eqref{eq:edp} by construction of $p$ and $p'$ with the principle of the compensation approach. We also define $f_1(v)=f(0,v)$ and $f_2(u)=f(u,0)$. By simple integration by parts, the PDE implies that the Laplace transforms $\hatf$, $\hatf_1$ and $\hatf_2$ satisfy the same functional equation~\eqref{eq:funceq} satisfied by $\hatpi$, i.e.
    \begin{equation*}
       \left[(x-y)^2+2(\delta_2-\delta_1)x+2(\delta_3-\delta_2)y \right]\hatf (x,y)
=\left(x-\frac{y}{2}\right)\hatf_1(y)+\left(y-\frac{x}{2}\right)\hatf_2(x) .
\end{equation*}
Corollary~\ref{cor:nuintermofp} implies that $\hatnu_1(y)=\hatf_1(y)$ and $\hatnu_2(x)=\hatf_2(x)$. Therefore, the functional equations satisfied by $\hatpi$ and $\hatf$ imply that $\hatpi (x,y)=\hatf (x,y)$ and we conclude that the density $\pi(u,v)=f(u,v)$. See \cite[Thm 5.1]{doetsch_introduction_1974} for the classical result on the injectivity of the Laplace transform.
\end{proof}

\appendix

\section{Appendix: Homogeneity relations}

\subsection{Homogeneity relations in the general case}

\begin{lemma}[Homogeneity relations, general case]
\label{lem:homogeneitygeneral}
Denote by $\hatpi( x, y;\lambda_1,\lambda_2)$, $\hatnu_1(y;\lambda_1,\lambda_2)$ and $\hatnu_2(x;\lambda_1,\lambda_2)$ the Laplace transforms associated to the parameters $\lambda_1$ and $\lambda_2$. Let us recall that 
\begin{equation*}
   \mu_1=\frac{\lambda_1}{\lambda_1+\lambda_2}
\quad\text{and}\quad
\mu_2=\frac{\lambda_2}{\lambda_1+\lambda_2}.
\end{equation*}
We have the homogeneity relations
\begin{equation*}
   \left\{ \begin{array}{rcl}
   \hatpi((\lambda_1+\lambda_2) x,(\lambda_1+\lambda_2) y;\lambda_1,\lambda_2) & = & \hatpi(x,y;\mu_1,\mu_2),\\
   \hatnu_1((\lambda_1+\lambda_2) y;\lambda_1,\lambda_2) & = &  (\lambda_1+\lambda_2)\hatnu_1(y;\mu_1,\mu_2),\\
   \hatnu_2((\lambda_1+\lambda_2) x;\lambda_1,\lambda_2) & = &  (\lambda_1+\lambda_2)\hatnu_2(x;\mu_1,\mu_2).
   \end{array}\right.
\end{equation*}
At the level of densities, it reads
\begin{equation*}
   \left\{ \begin{array}{rcl}
   \pi( x, y;\lambda_1,\lambda_2) & = &(\lambda_1+\lambda_2)^2 \pi((\lambda_1+\lambda_2) x,(\lambda_1+\lambda_2) y;\mu_1,\mu_2),\\
   \nu_1( y;\lambda_1,\lambda_2) & = & (\lambda_1+\lambda_2)^2 \nu_1((\lambda_1+\lambda_2) y;\mu_1,\mu_2),\\
   \nu_2( x;\lambda_1,\lambda_2) & = & (\lambda_1+\lambda_2)^2  \nu_2((\lambda_1+\lambda_2) x;\mu_1,\mu_2).
   \end{array}\right.
\end{equation*}
\end{lemma}
\begin{proof}
    An immediate computation starting from the functional equation yields 
\begin{multline}
\label{eq:comparing_1/2general}
\bigl((x-y)^2+\frac{\lambda_1}{\lambda_1+\lambda_2} x+\frac{\lambda_2}{\lambda_1+\lambda_2} y\bigr)\hatpi((\lambda_1+\lambda_2)x,(\lambda_1+\lambda_2) y;\lambda_1,\lambda_2)
\\
=
\left(x-\frac{y}{2}\right) \frac{\hatnu_1((\lambda_1+\lambda_2) y;\lambda_1,\lambda_2)}{(\lambda_1+\lambda_2)}+\left(y-\frac{x}{2}\right) \frac{\hatnu_2((\lambda_1+\lambda_2) x;\lambda_1,\lambda_2)}{(\lambda_1+\lambda_2)}.
\end{multline}
On the other hand,
\begin{equation}
\label{eq:comparing_2/2general}
\bigl((x-y)^2+\mu_1 x+\mu_2 y\bigr)\hatpi(x,y;\mu_1,\mu_2)
=
\left(x-\frac{y}{2}\right)\hatnu_1(y;\mu_1,\mu_2)+\left(y-\frac{x}{2}\right)\hatnu_2(x;\mu_1,\mu_2).
\end{equation}
Comparing \eqref{eq:comparing_1/2general} and \eqref{eq:comparing_2/2general}, noting that $\hatpi(0,0;\lambda_1,\lambda_2) = \hatpi(0,0;\mu_1,\mu_2)=1$ and recalling the uniqueness property for  the main functional equation~\eqref{eq:funceq} corresponding to a probability measure, stated at the end of Section~\ref{sec:1.1}, we deduce the first statement of the lemma concerning the Laplace transforms. The relations at the level of densities follow directly.
\end{proof}

\subsection{Homogeneity relations in the symmetric case}

In the symmetric case $\lambda_1=\lambda_2=\lambda$, we explain how to reduce to the case $\lambda=1$. This may help reduce to the number of parameters.
\begin{lemma}[Homogeneity relations, symmetric case]
\label{lem:homogeneity}
In the symmetric case, denote by $\hatpi( x, y;\lambda)$, $\hatnu_1(y;\lambda)$ and $\hatnu_2(x;\lambda)$ the Laplace transforms associated to the parameter $\lambda$. We have the homogeneity relations
\begin{equation}
\label{eq:homogeneity_Laplace}
   \left\{ \begin{array}{rcl}
   \hatpi(\lambda x,\lambda y;\lambda) & = & \hatpi(x,y;1),\\
   \hatnu_1(\lambda y;\lambda) & = & \lambda \hatnu_1(y;1),\\
   \hatnu_2(\lambda x;\lambda) & = & \lambda \hatnu_2(x;1).
   \end{array}\right.
\end{equation}
At the level of densities, it reads
\begin{equation*}
   \left\{ \begin{array}{rcl}
   \pi( x, y;\lambda) & = &\lambda^2 \pi(\lambda x,\lambda y;1),\\
   \nu_1( y;\lambda) & = & \lambda^2 \nu_1(\lambda y;1),\\
   \nu_2( x;\lambda) & = & \lambda^2  \nu_2(\lambda x;1).
   \end{array}\right.
\end{equation*}
\end{lemma}

\begin{proof}
First of all, the homogeneity relations on the densities are immediate consequences of the identities~\eqref{eq:homogeneity_Laplace} on the Laplace transforms, on which we therefore focus. A first direct proof of~\eqref{eq:homogeneity_Laplace} is obtained using the explicit formulas given in Corollary~\ref{thm:explicit_symmetric} and the main functional equation~\eqref{eq:funceq}, which in the symmetric case reads
\begin{equation*}
\bigl((x-y)^2+\lambda x+\lambda y\bigr)\hatpi(x,y;\lambda)
=
\left(x-\frac{y}{2}\right)\hatnu_1(y;\lambda)+\left(y-\frac{x}{2}\right)\hatnu_2(x;\lambda),
\end{equation*}
where we added $\lambda$ in our notation to emphasize the dependence on this parameter. 

We may now give a second approach for proving~\eqref{eq:homogeneity_Laplace}. 
An immediate computation starting from the above functional equation yields 
\begin{equation}
\label{eq:comparing_1/2}
\bigl((x-y)^2+ x+ y\bigr)\hatpi(\lambda x,\lambda y;\lambda)
=
\left(x-\frac{y}{2}\right) \frac{\hatnu_1(\lambda y;\lambda)}{\lambda}+\left(y-\frac{x}{2}\right) \frac{\hatnu_2(\lambda x;\lambda)}{\lambda}.
\end{equation}
On the other hand,
\begin{equation}
\label{eq:comparing_2/2}
\bigl((x-y)^2+ x+ y\bigr)\hatpi(x,y;1)
=
\left(x-\frac{y}{2}\right)\hatnu_1(y;1)+\left(y-\frac{x}{2}\right)\hatnu_2(x;1).
\end{equation}
Comparing \eqref{eq:comparing_1/2} and \eqref{eq:comparing_2/2}, and recalling the uniqueness property for  the main functional equation~\eqref{eq:funceq}, stated at the end of Section~\ref{sec:1.1}, 
we deduce that there exists a constant $\alpha>0$ such that
\begin{equation}
\label{eq:value_alpha}
   \left\{ \begin{array}{rcl}
   \hatpi(\lambda x,\lambda y;\lambda) & = &\alpha \hatpi(x,y;1),\\
   \hatnu_1(\lambda y;\lambda) & = &\alpha \lambda \hatnu_1(y;1),\\
   \hatnu_2(\lambda x;\lambda) & = &\alpha \lambda \hatnu_2(x;1).
   \end{array}\right.
\end{equation}
Evaluating \eqref{eq:value_alpha} at $x=y=0$ and using the normalization $\hatpi(0,0;\lambda) = \hatpi(0,0;1)=1$, one finds that $\alpha$ should be equal to $1$.
\end{proof}

\section{Some remarks on the function \texorpdfstring{$\theta_\mu$}{theta\_mu} of \texorpdfstring{\eqref{eq:def_theta_mu}}{(16)}}
\label{sec:app_theta_mu}

\subsection{A probabilistic interpretation of the function \texorpdfstring{$\theta_\mu$}{theta\_mu}}

Not surprisingly, the Jacobi theta-like function $\theta_\mu$ in \eqref{eq:def_theta_mu} admits a direct probabilistic interpretation (see \eqref{eq:probab_interpretation_theta_mu} below) in terms of Brownian motion conditioned to stay forever in the interval $[0,1]$. More specifically, for $t>0$ and $x,y\in(0,1)$, let $q_t(x,y)$ be the associated transition probability density. Using the recent results by Bougerol and Defosseux \cite[Eq.~(2.1)]{BoDe-22}, one has
\begin{equation*}
    q_t(x,y) = \frac{\sin(\pi y)}{\sin(\pi x)}e^{\pi^2 t/2}p_t(x,y),
\end{equation*}
where $p_t(x,y)$ is the transition probability density function of the killed Brownian motion in $[0,1]$, namely, 
\begin{equation}
\label{eq:killed_BM_[0,1]}
    p_t(x,y) = \frac{1}{2\sqrt{2\pi t}} \sum_{n\in\mathbb Z} \left(\exp\Bigl(-\frac{(x-y+2n)^2}{2t}\Bigr)-\exp\Bigl(-\frac{(x+y-2+2n)^2}{2t}\Bigr)\right),
\end{equation}
see Section~6 in Appendix~A.1 of \cite{Borodin-Salminen}.
As explained in \cite[Sec.~2.1]{BoDe-22}, it is actually possible to start the process at $x=0$ (using the idea of entrance density measure), and obtain the density function 
\begin{equation}
\label{eq:Poisson_q_t}
    q_t(0,y) = \lim_{x\to 0}q_t(x,y) = \sin(\pi y)\sum_{n\in\mathbb Z} n\sin(n\pi y)\exp\Bigl(-\pi^2 (n^2-1)\frac{t}{2}\Bigr),
\end{equation}
see \cite[Eq.~(2.5)]{BoDe-22}. The Jacobi transformation of our Lemma~\ref{lem:Jacobi_transform} leads directly to 
\begin{equation}
\label{eq:probab_interpretation_theta_mu}
    \theta_\mu(e^{-2/t}) = \frac{1}{\sin(\pi\mu)}\left(\frac{\pi t}{2}\right)^{3/2}\exp\left(\frac{\mu^2}{2t}-\frac{\pi^2 t}{2}\right)q_t(0,\mu).
\end{equation}
As a conclusion, up to a simple prefactor function, the theta function $\theta_\mu$ exactly describes the entrance density measure of the killed Brownian motion in $[0,1]$ starting from $0$.

The paper \cite{BoDe-22} by Bougerol and Defosseux contains a further interpretation of $q_t(0,\mu)$ (and thus of $\theta_\mu$ via \eqref{eq:probab_interpretation_theta_mu})\ as a space-time non-negative harmonic function for a killed Brownian motion in a certain affine cone. We shall not elaborate on this connection here, except to say that it is natural to expect a strong link between our model and space-time Brownian motion, as suggested by our Equation~\eqref{eq:gapSRBM}, the starting point of this entire investigation. 

\subsection{Connection with the Ramanujan theta function}

The Ramanujan theta function is classically defined for $a,b\in\mathbb C$ such that $\vert ab\vert<1$ by 
\begin{equation*}
    f(a,b) = \sum_{n\in\mathbb Z} a^{\frac{n(n+1)}{2}}b^{\frac{n(n-1)}{2}}.
\end{equation*}
If we introduce
\begin{equation*}
   g(a,b)  = \sum_{n\in\mathbb Z}  na^{\frac{n(n+1)}{2}}b^{\frac{n(n-1)}{2}}= \Bigl(a\frac{\partial}{\partial a}-b\frac{\partial }{\partial b}\Bigr)f(a,b), 
\end{equation*}
then the function of \eqref{eq:def_theta_mu} can be expressed as 
\begin{equation*}
    \theta_\mu(q) = g\bigl(q^{1+\mu},q^{1-\mu}\bigr)+\frac{\mu}{2}f\bigl(q^{1+\mu},q^{1-\mu}\bigr).
\end{equation*}
This connection is not central for our purpose, but is nevertheless interesting to observe.  

\subsection*{Acknowledgments} 
This project has received funding from the 
European Research Council (ERC)\ under the European Union's Horizon 2020 research and innovation programme under the Grant Agreement No.~759702, from the ANR RESYST (ANR-22-CE40-0002), from the National Science Foundation under Grant DMS-20-04997 and DMS-20-08427, and from Centre Henri Lebesgue, programme ANR-11-LABX-0020-0. SF and KR would like to thank Manon Defosseux, Andrew Elvey Price and Timothy Huber for interesting discussions related to theta-functions.

\bibliographystyle{apalike}

\end{document}